\documentclass[onefignum,onetabnum]{siamart251216}

\usepackage{lipsum}
\usepackage{amsfonts}
\usepackage{graphicx}
\usepackage{epstopdf}
\usepackage{algorithm}
\usepackage{algpseudocode}

\ifpdf
  \DeclareGraphicsExtensions{.eps,.pdf,.png,.jpg}
\else
  \DeclareGraphicsExtensions{.eps}
\fi

\newsiamremark{remark}{Remark}
\newsiamremark{hypothesis}{Hypothesis}
\crefname{hypothesis}{Hypothesis}{Hypotheses}
\newsiamthm{claim}{Claim}
\newsiamthm{example}{Example}
\newsiamthm{property}{Property}
\newsiamremark{fact}{Fact}
\crefname{fact}{Fact}{Facts}

\headers{An Example Article}{J. Li, S. Liu, J. Yang, and T. Tang}

\author{
Jianan Li\thanks{Department of Mathematics, Southern University of Science and Technology, Shenzhen, China (\email{11930493@mail.sustech.edu.cn}).}
\and
Shuang Liu\thanks{Department of Mathematics, Southern University of Science and Technology, Shenzhen, China (\email{12031031@mail.sustech.edu.cn}).}
\and
Tao Tang\thanks{School of Mathematics and Statistics, Guangzhou Nanfang College, Guangzhou, China; Zhuhai SimArk Technology Co., LTD, Zhuhai, Guangdong, China. (\email{ttang@nfu.edu.cn}).}
\and
Jiang Yang\thanks{Corresponding author: Department of Mathematics, Guangdong Provincial Key Laboratory of Computational Science and Material Design, SUSTech International Center for Mathematics \& National Center for Applied Mathematics Shenzhen (NCAMS), Southern University of Science and Technology, Shenzhen, China (\email{yangj7@sustech.edu.cn}).}
}

\usepackage{amsopn}

\usepackage{tabularx}
\usepackage{booktabs}
\usepackage{amssymb}
\usepackage{amsmath}
\usepackage{color}
\usepackage{float}
\usepackage{subfig}
\usepackage{subfloat}
\usepackage{algpseudocode}

\newcommand{\bm}{\boldsymbol}

\newcommand{\TheTitle}{A second-order product-type implicit-explicit Runge-Kutta method preserving unit length and energy dissipation structures for gradient flows of vector fields}
\title{{\TheTitle}\thanks{This is the Author's Accepted Manuscript (AAM) of an article accepted for publication in \textit{SIAM Journal on Scientific Computing}.}}
\ifpdf
\hypersetup{
  pdftitle={\TheTitle},
  pdfauthor={}
}
\fi

\begin{document}

\maketitle

\begin{abstract}
Gradient flows of unit vector fields arise in a wide range of physical models such as harmonic map heat flows, nematic liquid crystals, and magnetization dynamics. Designing numerical schemes that simultaneously preserve the unit length constraint and dissipate energy is essential for reliable simulations of such systems. Although projection methods can effectively enforce the unit length constraint, ensuring energy dissipation under projection, especially in high-order schemes, remains challenging. Unlike traditional implicit-explicit Runge-Kutta (IMEX-RK) methods, in this work we propose a general methodology for constructing product-type IMEX-RK schemes that offers greater adaptability to various models with the goal of designing structure-preserving numerical schemes. For gradient flows of unit vector fields with Dirichlet energy, we design a linear and second-order numerical scheme that simultaneously preserves energy dissipation and the unit length constraint by using product-type IMEX-RK methods and projection techniques. Numerical experiments verify the accuracy, stability, and structure-preserving properties of the scheme. According to our best knowledge, this is the first second-order linear scheme that can preserve both the unit length and the original Dirichlet energy for harmonic map heat flows.
\end{abstract}



\section{Introduction}
	Gradient flows of vector fields with unit length constraint are fundamental models in many areas of mathematical physics. Such flows arise as gradient flows constrained on nonlinear manifolds, most commonly the unit sphere, and describe the relaxation dynamics toward energy minimizers under physical constraints. Notable examples include the harmonic map heat flow \cite{brezis1986harmonic, hardt1986existence, riviere1995everywhere, bartels2005stability, gui2022convergence}, the Oseen-Frank model for liquid crystals \cite{oseen1933theory, frank1958liquid, cohen1987minimum, Alouges1997, ramage2013preconditioned, virga2018variational}, and the Landau–Lifshitz–Gilbert (LLG) equation modeling magnetization dynamics \cite{landau1935theory, landau1992theory, weinan2001numerical, bartels2008numerical, alouges2014convergent}. 
    These systems are governed by energy functionals and possess rich geometric structures, such as energy dissipation and constraint preservation, which must be carefully respected in numerical simulations.

    A representative example is the gradient flow of the Dirichlet energy
    \[
	E[\bm{m}(\bm{x})] = \int_{\Omega} \frac{1}{2}|\nabla \bm{m}|^2 \, \mathrm{d}\bm{x} 
    \]
    under unit length constraint, where $\bm{m}:\Omega\to\mathbb{R}^{3}$ is a vector field satisfying $|\bm{m}(\bm{x})|=\sqrt{\bm{m}\cdot\bm{m}}=1$. 
    The evolution equation is given by
    \begin{equation}\label{Gradient Flow}
        \begin{aligned}
            &\bm{m}_t = - P(\bm{m})\bm{\mu}, \\
		&\bm{\mu} = \frac{\delta E}{\delta \bm{m}} = -\Delta \bm{m},  \\ 
		&P(\bm{m}) =  \alpha\left(I - \frac{\bm{m}\bm{m}^\top}{|\bm{m}|^2}\right) + \beta \bm{m}\times \cdot, \\
        &|\bm{m}(\bm{x},0)| = 1,
        \end{aligned}
    \end{equation}
	where $\alpha>0,\,\beta\in\mathbb{R}$ are constants and the system is subject to time-independent Dirichlet or homogeneous Neumann boundary conditions. The operator $P(\bm{m})$ ensures $\bm{m}\cdot\bm{m}_t=0$, and thus $|\bm{m}(\bm{x},t)|=1$ is preserved for all time. The term $\beta$ involved can be interpreted as a Hamiltonian (or symplectic) flow part without energy dissipation.

    By analyzing equation \eqref{Gradient Flow}, two fundamental properties of the flow can be derived:
\begin{itemize}
    \item Energy dissipation: the total energy decreases over time:
\[
\frac{d}{dt} E[\bm{m}(\bm{x})] = -\alpha \int_{\Omega} \bm{\mu}^\top \mathbf{P}(\bm{m}) \bm{\mu} \, \mathrm{d}\bm{x} \leqslant 0.  
\]
    \item The unit length constraint is maintained:
\[
\frac{d}{dt} |\bm{m}|^2 / 2 = -\bm{m}^\top  \mathbf{P}(\bm{m}) \bm{\mu} = 0.
\]
\end{itemize}
For numerical methods, preserving these two properties at the discrete level is crucial for long-time stability and physical fidelity. However, achieving both in a structure-preserving and high-order scheme is nontrivial.

Existing numerical methods to enforce the unit length constraint \( |\bm{m}| = 1 \) can be broadly categorized into two main approaches. The first involves relaxing the constraint by incorporating penalty terms or introducing Lagrange multipliers, as demonstrated in works such as \cite{pistella1999numerical, prohl2001computational, adler2015energy, cheng2022new2}. The second and most widely used method is the projection method \cite{Alouges1997, wang2001gauss, Alouges2008, chen2021convergence, an2021optimal, gui2022convergence, li2020arbitrarily}, which combines flexibility with ease of implementation. Projection methods are compatible with standard spatial discretization techniques, such as finite difference methods (FDMs) or finite element methods (FEMs), and linearly implicit time-stepping schemes. Although projection methods easily handle the unit length constraint, developing high-order schemes that simultaneously preserve energy dissipation remains a significant challenge.

In the literature, only a few methods successfully achieve this dual preservation. For instance, a second-order finite difference scheme was proposed in \cite{Fuwa2012} that automatically preserves both the unit length and energy dissipation without utilizing projection. However, this approach requires solving a computationally expensive, fully coupled nonlinear system at each time step and suffers from a CFL condition $\tau\leqslant Ch^2$. To alleviate this computational burden, recent efforts have focused on decoupling strategies. For example, within the projection framework, a first-order linear semi-implicit projection method (SIP1) was designed in \cite{du2024} to preserve both structures. To achieve higher-order accuracy without fully coupled nonlinear systems, the second-order Lagrange multiplier (LM2) method \cite{cheng2023length} was introduced. Notably, while LM2 significantly reduces complexity by solving only a scalar nonlinear equation for the multiplier, this residual nonlinear step still poses potential solvability issues. This fundamental limitation motivates the development of our scheme. By completely avoiding nonlinear solvers, our proposed method achieves the same dual structure preservation through purely linear updates, thereby unconditionally guaranteeing unique solvability at each time step.
	
To systematically achieve this high-order accuracy within a linear framework, our scheme draws fundamental inspiration from Runge-Kutta (RK) methods.
RK methods are widely used approach to design high-order numerical schemes, such as implicit-explicit Runge-Kutta (IMEX-RK) methods \cite{Ascher1997,Bos2007,Jang2015,Gui2024,Fu2024ED} and exponential time differencing Runge-Kutta (ETDRK) methods \cite{Beylkin1998,cox2002,du2019,Cao2024}. In recent years, several energy-dissipating Runge-Kutta methods have been developed, including ETDRK methods \cite{Fu2022, Fu2024},
IMEX-RK methods \cite{Fu2024ED, Li2024}, and extrapolated RK-SAV methods \cite{Akrivis2019, Tang2022}. Unlike traditional additive IMEX-RK schemes \cite{Ascher1997}, \cite{Li2024} proposed a product form for IMEX-RK schemes.

In this work, we propose a general methodology for constructing product-type IMEX-RK (PRK) schemes.
Compared to traditional IMEX-RK methods \cite{Ascher1997}, the PRK methods exhibit significantly enhanced flexibility. This enables the tailoring of distinct numerical schemes to specific models, thereby presenting expanded opportunities for designing structure-preserving numerical schemes.
For gradient flows of unit vector fields with Dirichlet energy, we design a linear and second-order time-discrete numerical scheme that simultaneously preserves energy dissipation and the unit length constraint by using PRK methods and projection techniques. A sufficient condition to ensure energy dissipation after projection onto the unit sphere is that the length is not less than $1$ before projection \cite{du2024}. Therefore, we propose certain requirements for the IMEX-RK coefficients to ensure energy dissipation and length increase before projection. However, the existence of a third-order IMEX-RK scheme that meets such requirements is open.
To the best of our knowledge, this is the first linear second-order scheme that unconditionally preserves both energy dissipation and the unit length constraint for gradient flow \eqref{Gradient Flow}.  

The rest of this paper is organized as follows. Section \ref{sec2} introduces the proposed product-type IMEX-RK (PRK) framework and derives the order conditions. Section \ref{sec3} presents the PRK scheme for gradient flows of unit vector fields and proves the energy dissipation and preservation of the unit length constraint. In Section \ref{sec4}, we present numerical experiments to validate the accuracy, efficiency, and robustness of the method.

\section{Product-type IMEX-RK methods}\label{sec2}
IMEX-RK methods are designed to handle problems where both stiff and non-stiff components coexist. 
In this formulation, the stiff components are treated implicitly to ensure stability without restrictive time step constraints, while the non-stiff components are handled explicitly to preserve accuracy and efficiency. This partitioning allows an optimal balance between computational cost and numerical performance. Their flexibility and efficiency make them an attractive choice for a wide range of applications in computational science and engineering.

In this section, we introduce the traditional IMEX-RK methods. Next, we propose a general methodology for constructing product-type IMEX-RK schemes and derive the order conditions.


\subsection{Implicit-explicit Runge-Kutta (IMEX-RK) methods}
When dealing with some practical problems, the differential equations can be
written as the sum of two terms, a stiff one and a non-stiff one:
\begin{equation}
	\label{addform}
	u_t = g(u)+f(u),
\end{equation}
where $f$ corresponds to the non-stiff term and $g$ corresponds to the stiff term, such as reaction-diffusion equations, convection-diffusion equations and hyperbolic systems with relaxation \cite{Caflisch97, Kennedy2003}.

We now introduce IMEX-RK schemes \cite{Ascher1997}. 
An $s$-stage diagonally implicit Runge-Kutta (DIRK) scheme for the stiff term $g(u)$ is determined by a coefficient matrix 
$A=\left(a_{i j}\right) \in \mathbb{R}^{s \times s}$, and column vectors $\boldsymbol{c} = (c_i)$, $\boldsymbol{b}=(b_i) \in \mathbb{R}^{s}$ in the usual Butcher notation, where $\sum_{j=1}^i a_{i j}=c_i$. For the non-stiff term $f(u)$, an $s$-stage explicit scheme is determined by a different coefficient matrix $\widehat{A} =\left(\hat{a}_{i j}\right) \in \mathbb{R}^{s \times s}$, and column vectors $\boldsymbol{\widehat{c}}= (\hat{c}_i)$, $\boldsymbol{\widehat{b}}= (\hat{b}_i) \in \mathbb{R}^{s}$, where $\sum_{j=1}^i \hat{a}_{i j}=\hat{c}_i$ . 
One step from $u^{n}$ to $u^{n+1}$ of the IMEX-RK scheme is given as follows
\begin{equation}
	\label{addformIMEX}
	\begin{aligned}
		&u_0  =u^n, \\
		&u_i  =u_0+\tau\Big(\sum_{j=1}^i a_{i j} g(u_j)+\sum_{j=1}^i \hat{a}_{i j} f\left(u_{j-1}\right)\Big), \quad \text{for }i=1,\,\dots,\, s, \\
		&u^{n+1} =u_0+\tau\Big(\sum_{i=1}^s b_{i} g(u_i)+\sum_{i=1}^s \hat{b}_{i} f\left(u_{i-1}\right)\Big).
	\end{aligned}
\end{equation}
The Butcher tableau \cite{butcher2016numerical} of IMEX-RK schemes is given as follows,
\begin{equation}
	A:\,\begin{tabular}{c|ccccc} $c_1$ &  $a_{11}$ & 0 & $\ldots$ & 0 \\ $c_2$ & $a_{21}$ & $a_{22}$ & $\ddots$ & $\vdots$ \\ $\vdots$ &  $\vdots$ &  & $\ddots$ & $0$ \\ $c_s$ &  $a_{s 1}$ & $a_{s 2}$ & $\ldots$ & $a_{s s}$ \\ \hline &  $b_1$ & $b_2$ & $\ldots$ & $b_s$ \end{tabular}
	\qquad 
	\widehat{A}:\,\begin{tabular}{c|ccccc}  $\hat{c}_1$ & $\hat{a}_{11}$ & 0 & $\ldots$ & 0  \\ $\hat{c}_2$ & $\hat{a}_{21}$ & $\hat{a}_{22}$ & $\ddots$ & $\vdots$  \\ $\vdots$ & $\vdots$ &  & $\ddots$ & $0$\\ $\hat{c}_s$ & $\hat{a}_{s 1}$ & $\hat{a}_{s 2}$ & $\ldots$ & $\hat{a}_{s s}$  \\ \hline & $\hat{b}_1$ & $\hat{b}_2$ & $\ldots$ & $\hat{b}_s$  \end{tabular}.
\end{equation}
Note that $a_{ij},\hat{a}_{ij}=0$ for $j>i$.

However, the stiff terms may not appear in the equation as an additive form like \eqref{addform}. Therefore, it is not easy to apply the standard additive IMEX-RK schemes directly. A more general class of such problems can be written as 
\begin{equation*}
	\begin{aligned}
		u_t = \mathcal{H}\left(u/\varepsilon,u\right),
	\end{aligned}
\end{equation*}
where the small parameter $\varepsilon$ represent the stiff term. In \cite{Bos2016}, the authors propose a more general IMEX-RK scheme for such problems.

\subsection{Product-type IMEX-RK (PRK) methods}
Compared to the additive IMEX-RK schemes, \cite{Bos2016} offers a broader framework. However, its generality can also be a limitation as it may lack the necessary flexibility for analysis and construction under specific constraints. Therefore, we propose a methodology for constructing product-type IMEX-RK (PRK) schemes, which retains a generalized structure while ensuring sufficient flexibility. More importantly, this approach provides a framework for designing coefficient matrices across different problems, facilitating the development of new schemes that adhere to specific physical constraints. For example, for the anisotropic phase-field dendritic growth model, PRK methods break the barrier of the lack of high-order schemes preserving the original energy dissipation \cite{Li2024}.

Firstly, we start from a simple product form equation
\begin{equation*}
	u_t = f_1(u)f_2(u),
\end{equation*}
where $f_1$ is the non-stiff term and $f_2$ is the stiff term. Therefore, we discretize $f_1$ explicitly and $f_2$ implicitly. 
We propose a method for designing IMEX-RK schemes for such equations.
Consider the following Butcher tableau
\begin{equation*}
	\begin{aligned}
		A:\,\begin{array}{c|cccc}
			c_1 &a_{11}&0&\cdots &0\\
			c_2 &a_{21}&a_{22}&\ddots &\vdots\\
			\vdots&\vdots&&\ddots&0\\
			c_s &a_{s1}&a_{s2}&\cdots &a_{ss}\\ \hline
			&b_1&b_2&\cdots &b_s
		\end{array},\,
		D^{(r)}:\,\begin{array}{c|cccc}
			1 &d_{11}^{(r)}&0&\cdots &0\\
			1 &d_{21}^{(r)}&d_{22}^{(r)}&\ddots &\vdots\\
			\vdots&\vdots&&\ddots&0\\
			1 &d_{s1}^{(r)}&d_{s2}^{(r)}&\cdots &d_{ss}^{(r)}\\ \hline
			&&& &
		\end{array},\;r=1,\,2,\,\dots,
	\end{aligned}
\end{equation*}
where $c_i=\sum_{j=1}^{i} a_{ij}$
and $1=\sum_{j=1}^{i} d_{ij}^{(r)}$. Let $A=(a_{ij})_{s\times s}$, $D^{(r)}=\left(d_{ij}^{(r)}\right)_{s\times s}$, $\bm{b}=(b_1,\,\dots,\,b_s)^\top$ and $\bm{c}=(c_1,\,\dots,\,c_s)^\top$.

Given the numerical solution $u^n$ at time $t^n$. The PRK scheme can be given as follows:
\begin{equation}
	\label{intrPRK1}
	\begin{aligned}
		&u_0  =u^n, \\
		&u_i  =u_0+\tau\sum_{j=1}^i a_{i j} f_1\Big(\sum_{k=1}^j d_{jk}^{(1)}u_{k-1}\Big) f_2\Big(\sum_{k=1}^j d_{jk}^{(2)}u_{k}\Big), \quad \text{for }i=1,\,\dots,\, s, \\
		&u^{n+1} =u_0+\tau\sum_{j=1}^s b_{ j} f_1\Big(\sum_{k=1}^j d_{jk}^{(1)}u_{k-1}\Big)f_2\Big(\sum_{k=1}^j d_{jk}^{(2)}u_{k}\Big).
	\end{aligned}
\end{equation}
We can also choose whether the coefficient matrices $D^{(r)}$ are `inside' or `outside' of $f_r(u)$ in the IMEX-RK scheme. For example, the PRK scheme can also be given as follows:
\begin{equation}\label{intrPRK2}
	\begin{aligned}
		&u_0  =u^n, \\
		&u_i  =u_0+\tau\sum_{j=1}^i a_{i j}\Big( \sum_{k=1}^j d_{jk}^{(1)}f_1(u_{k-1}) \Big)\Big(\sum_{k=1}^j d_{jk}^{(2)}f_2(u_{k})\Big), \quad \text{for }i=1,\,\dots,\, s, \\
		&u^{n+1} =u_0+\tau\sum_{j=1}^s b_{ j}\Big( \sum_{k=1}^j d_{jk}^{(1)}f_1(u_{k-1})\Big)\Big( \sum_{k=1}^j d_{jk}^{(2)}f_2(u_{k})\Big).
	\end{aligned}
\end{equation}

We use a more compact form introduced in \cite{Li2024} to represent the PRK scheme.
Let 
\begin{gather*}
	\bm{u} := \begin{pmatrix}
		u_1 \\ u_2\\ \vdots \\ u_s
	\end{pmatrix},   \;
	\widehat{\bm{u}} := \begin{pmatrix}
		u_0 \\ u_1\\ \vdots \\ u_{s-1}
	\end{pmatrix},  \;
	f_r(\bm{w}) = \begin{pmatrix}
		f_r(w_1) \\ f_r(w_2)\\ \vdots \\ f_r(w_s)
	\end{pmatrix},\;
	\bm{1} := \begin{pmatrix}
		1 \\ 1\\ \vdots \\ 1
	\end{pmatrix},\;
	 \bm{0} := \begin{pmatrix}
		0 \\ 0\\ \vdots \\ 0
	\end{pmatrix},
\end{gather*}
where $\bm{w}=(w_1,\,\dots,\,w_s)^\top$. Then the PRK scheme \eqref{intrPRK1} can be written as follows:
\begin{equation}
	\begin{aligned}
		&\bm{u} = u_0\bm{1} + \tau A (f_1(D^{(1)}\widehat{\bm{u}})\circ f_2(D^{(2)}\bm{u})),\\
		&u^{n+1} = u_0 + \tau \bm{b}^\top(f_1(D^{(1)}\widehat{\bm{u}})\circ f_2(D^{(2)}\bm{u})),
	\end{aligned}
\end{equation}
where the symbol `$\circ$' represents the Hadamard product, i.e., the element-wise product of vectors.

Now, we propose a methodology for constructing PRK schemes. First, we split the equation into several parts. 
For each part, design an `outside' RK coefficient matrix $A^{(i)}$ with $A^{(i)}\bm{1}=\bm{c}^{(i)}$, $\bm{b}^{(i)}=\left(b_j^{(i)}\right)_{s\times 1}$,  and several `inside' RK coefficient matrices $D^{(i,r)}$ with $D^{(i,r)}\bm{1}=\bm{1}$: 
\begin{equation*}\renewcommand\arraystretch{1.5}
	\begin{aligned}
		A^{(i)}:\,\begin{array}{c|cccc}
			c_1^{(i)} &a_{11}^{(i)}&0&\cdots &0\\
			c_2^{(i)} &a_{21}^{(i)}&a_{22}^{(i)}&\ddots &\vdots\\
			\vdots&\vdots&&\ddots&0\\
			c_s^{(i)} &a_{s1}^{(i)}&a_{s2}^{(i)}&\cdots &a_{ss}^{(i)}\\ \hline
			&b_1^{(i)}&b_2^{(i)}&\cdots &b_s^{(i)}
		\end{array},\,
		D^{(i,r)}:\,\begin{array}{c|cccc}
			1 &d_{11}^{(i,r)}&0&\cdots &0\\
			1 &d_{21}^{(i,r)}&d_{22}^{(i,r)}&\ddots &\vdots\\
			\vdots&\vdots&&\ddots&0\\
			1 &d_{s1}^{(i,r)}&d_{s2}^{(i,r)}&\cdots &d_{ss}^{(i,r)}\\ \hline
			&&& &
		\end{array},
	\end{aligned}
\end{equation*}
for $i=1,\,2,\,\dots$.
We call IMEX-RK methods of this type \textbf{product-type IMEX-RK (PRK)} methods. 
For example, if an equation can be split into
\begin{equation*}
	u_t = M(u)(\mathcal{L}u+\mathcal{N}(u)),
\end{equation*}
where $M(u)$ is a nonnegative mobility, $\mathcal{L}u$ is a linear stiff term, and $\mathcal{N}(u)$ is a nonlinear non-stiff term, we can design a PRK scheme as follows:
\begin{equation*}
    \begin{aligned}
        &u_0=u^n,\\
        &\bm{u} = u_0\bm{1} + \tau \left[A^{(1)} (M(D^{(1,1)}\widehat{\bm{u}})\circ D^{(1,2)}\mathcal{L}\bm{u})+A^{(2)} (M(D^{(2,1)}\widehat{\bm{u}})\circ D^{(2,2)}\mathcal{N}(\widehat{\bm{u}}))\right],\\
        &u^{n+1} = u_0 + \tau \left[\bm{b}^{(1)T} (M(D^{(1,1)}\widehat{\bm{u}})\circ D^{(1,2)}\mathcal{L}\bm{u})+\bm{b}^{(2)T} (M(D^{(2,1)}\widehat{\bm{u}})\circ D^{(2,2)}\mathcal{N}(\widehat{\bm{u}}))\right].
    \end{aligned}
\end{equation*}
In addition, PRK schemes can be applied in coupled multiphysics systems to design structure-preserving schemes, such as the anisotropic phase-field dendritic growth model \cite{Li2024}.

\subsection{Order conditions and local truncation error}
In this subsection, we calculate the order conditions and local truncation error of PRK schemes. For convenience, we only consider PRK schemes of the form \eqref{intrPRK1}.
Define
\begin{equation*}
	J = \begin{pmatrix}
		0&&&\\
		1&0&&\\
		&\ddots&\ddots&\\
		&&1&0
	\end{pmatrix}
\end{equation*}

First, we give an \textbf{a priori estimate} of the PRK scheme \eqref{intrPRK1}.
Assume that the asymptotic expansion of $\bm{u}$ is given as follows
\begin{equation}
	\label{ocPriorasy}
	\bm{u} = \bm{u}^{(0)}+\tau \bm{u}^{(1)}+ \tau^2 \bm{u}^{(2)} +O(\tau^3),
\end{equation}
where $\bm{u}^{(i)} $ are functions of $u_0$ for $i=0,1,2$.
Applying Taylor expansion to $f_1(D^{(1)}\hat{\bm{u}})$ and $f_2(D^{(2)}\bm{u})$, we have
\begin{equation}
	\label{TaylorEMN}
	\begin{aligned}
		f_1(D^{(1)}\hat{\bm{u}}) =& f_1(u_0)\bm{1} +  f_1'(u_0)D^{(1)}J(\bm{u}-u_0\bm{1}) + \frac{1}{2}  f_1''(u_0)[D^{(1)}J(\bm{u}-u_0\bm{1})]^2+ O(\tau^3),\\
		f_2(D^{(2)}\bm{u}) =& f_2(u_0)\bm{1} +  f_2'(u_0)D^{(2)}(\bm{u}-u_0\bm{1}) + \frac{1}{2}  f_2''(u_0)[D^{(2)}  (\bm{u}-u_0\bm{1})]^2+ O(\tau^3),
	\end{aligned}
\end{equation}
where we use $\bm{w}^2$ to represent $\bm{w}\circ \bm{w}$. Therefore,
\begin{equation*}
	\begin{aligned}
		\bm{u} =& u_0\bm{1} +\tau f_1(u_0)f_2(u_0)A\bm{1}+\tau Af_1'(u_0)D^{(1)}J(\bm{u}-u_0\bm{1})f_2(u_0)\\
		&+\tau A f_1(u_0)f_2'(u_0)D^{(2)}(\bm{u}-u_0\bm{1}) + O(\tau^3).
	\end{aligned}
\end{equation*}
Then inserting \eqref{ocPriorasy} into the above equation, we have
\begin{equation*}
	\begin{aligned}
		\bm{u}=& u_0\bm{1}+ \tau f_1(u_0)f_2(u_0)A\bm{1}\\
		& +\tau^2 \left[Af_1'(u_0)D^{(1)}J\bm{u}^{(1)}f_2(u_0)+ A f_1(u_0)f_2'(u_0)D^{(2)}\bm{u}^{(1)}\right]+O(\tau^3).
	\end{aligned}
\end{equation*}
Note that 
\begin{equation*}
	\begin{aligned}
		\bm{u}^{(0)} =& u_0\bm{1},\\
		\bm{u}^{(1)} =& f_1(u_0)f_2(u_0)A\bm{1},\\
		\bm{u}^{(2)} = & Af_1'(u_0)D^{(1)}J\bm{u}^{(1)}f_2(u_0)+ A f_1(u_0)f_2'(u_0)D^{(2)}\bm{u}^{(1)}.
	\end{aligned}
\end{equation*}

Next, we calculate the local truncation error of the PRK scheme \eqref{intrPRK1}. Define $v_0 = u(t_n)$ and 
\begin{equation}
	\label{ocasyv}
	\begin{aligned}
		&\bm{v} = \bm{v}^{(0)}+\tau \bm{v}^{(1)}+ \tau^2 \bm{v}^{(2)},\\
		&\hat{\bm{v}} = \bm{v}^{(0)}+\tau J\bm{v}^{(1)}+ \tau^2 J\bm{v}^{(2)} ,
	\end{aligned}
\end{equation}
where 
\begin{equation}
	\begin{aligned}
		\bm{v}^{(0)} =& v_0\bm{1},\\
		\bm{v}^{(1)} =& f_1(v_0)f_2(v_0)A\bm{1},\\
		\bm{v}^{(2)} = & Af_1'(v_0)D^{(1)}J\bm{v}^{(1)}f_2(v_0)+ A f_1(v_0)f_2'(v_0)D^{(2)}\bm{v}^{(1)}.
	\end{aligned}
\end{equation}
Then, the local truncation error of the PRK scheme \eqref{intrPRK1} $\bm{\eta}$ and $\eta_n$ can be define as follows
\begin{equation}
	\label{localTE}
	\begin{aligned}
		&v_0=u(t_n)\\
		&\bm{v} = v_0\bm{1}+\tau A[f_1(D^{(1)}\widehat{\bm{v}})\circ f_2(D^{(2)}\bm{v})]+\tau \bm{\eta},\\
		&u(t_{n+1}) = v_0+\tau \bm{b}^\top[f_1(D^{(1)}\widehat{\bm{v}})\circ f_2(D^{(2)}\bm{v})]+ \tau \eta_n.
	\end{aligned}
\end{equation}
Based on the above a priori estimate, we can obtain $\bm{\eta}=O(\tau^2)$. Here, $\eta_n$ is defined as the rate form of the local truncation error, so that the incremental error is $\tau \eta_n = O(\tau^4)$ for a third-order method.

Now, we calculate the third-order conditions of the PRK scheme \eqref{intrPRK1}, i.e., $\eta_n=O(\tau^3)$. Using the Taylor expansion to $f_1(D^{(1)}\hat{\bm{u}})$ and $f_2(D^{(2)}\bm{u})$ as shown in \eqref{TaylorEMN}, we have
\begin{equation*}
	\begin{aligned}
		u(t_{n+1})=& v_0 +\tau f_1(v_0)f_2(v_0)\bm{b}^\top\bm{1}+\tau \bm{b}^\top f_1'(v_0)D^{(1)}J(\bm{v}-v_0\bm{1})f_2(v_0)\\
		&+\tau \bm{b}^\top f_1(v_0)f_2'(v_0)D^{(2)}(\bm{v}-v_0\bm{1}) \\
		&+\tau \bm{b}^\top [f_1'(v_0)D^{(1)}J(\bm{v}-v_0\bm{1})\circ f_2'(v_0)D^{(2)}(\bm{v}-v_0\bm{1})]\\
		&+\tau \frac{1}{2} \bm{b}^\top  f_1''(v_0)[D^{(1)}J(\bm{v}-v_0\bm{1})]^2f_2(v_0)\\
		&+\tau \frac{1}{2} \bm{b}^\top  f_1(v_0) f_2''(v_0)[D^{(2)}(\bm{v}-v_0\bm{1})]^2 + O(\tau^4)+\tau \eta_n.
	\end{aligned}
\end{equation*}
Therefore, 
\begin{equation*}
	\begin{aligned}
		u(t_{n+1})=&v_0+ \tau f_1(v_0)f_2(v_0)\bm{b}^\top \bm{1}\\
		& +\tau^2 \bm{b}^\top \left[f_1'(v_0)f_2(v_0)D^{(1)}J\bm{v}^{(1)}+ f_1(v_0)f_2'(v_0)D^{(2)}\bm{v}^{(1)}\right]\\
        & +\tau^3 \bm{b}^\top \left[f_1'(v_0)f_2(v_0)D^{(1)}J\bm{v}^{(2)}+ f_1(v_0)f_2'(v_0)D^{(2)}\bm{v}^{(2)}\right.\\
        &+ \frac{1}{2} f_1''(v_0)[D^{(1)}J\bm{v}^{(1)}]^2f_2(v_0)+f_1'(v_0)f_2'(v_0)D^{(1)}J\bm{v}^{(1)}\circ D^{(2)}\bm{v}^{(1)}\\
		&+ \left.\frac{1}{2} f_1(v_0) f_2''(v_0)[D^{(2)}\bm{v}^{(1)}]^2 \right]+ O(\tau^4)+\tau \eta_n.
	\end{aligned}
\end{equation*}
Then, comparing with  
\begin{equation*}
	u(t_{n+1}) = u(t_n)+\tau u_t(t_n)+\frac{1}{2}\tau^2 u_{tt}(t_n)+\frac{1}{6}\tau^3 u_{ttt}(t_n)+O(\tau^4),
\end{equation*}
where 
\begin{equation*}
	\begin{aligned}
		u_t =& f_1(u)f_2(u),\\
		u_{tt}=& f_1'(u)f_2(u)u_t+f_1(u)f_2'(u)u_t,\\
		u_{ttt}=&\left [f_1'(u)f_2(u)+f_1(u)f_2'(u) \right]u_{tt} \\
		& +\left[f_1''(u)f_2(u)+2f_1'(u)f_2'(u)+f_1(u)f_2''(u) \right]u_t^2,
	\end{aligned}
\end{equation*}
we obtain the order conditions as follows:
\begin{itemize}
\item \textbf{First-order condition:}
\begin{equation}\label{prk_1stordcon}
	\bm{b}^\top \bm{1}=1.
\end{equation}
\item \textbf{Second-order condition:}
\begin{equation}\label{prk_2ndordcon}
	\begin{aligned}
		\bm{b}^\top D^{(1)}J\bm{c}=\bm{b}^\top D^{(2)}\bm{c}=\frac{1}{2}.
	\end{aligned}
\end{equation}
\item  \textbf{Third-order condition:}
\begin{equation}\label{prk_3rdordcon}
	\begin{gathered}
		\bm{b}^\top D^{(1)}JAD^{(1)}J\bm{c}=\bm{b}^\top D^{(1)}JAD^{(2)}\bm{c}=\frac{1}{6},\\
        \bm{b}^\top D^{(2)}AD^{(1)}J\bm{c}=\bm{b}^\top D^{(2)}AD^{(2)}\bm{c}=\frac{1}{6},\\
		\bm{b}^\top (D^{(1)}J\bm{c})^2=\bm{b}^\top (D^{(1)}J\bm{c}\circ D^{(2)}\bm{c})=\bm{b}^\top (D^{(2)}\bm{c})^2=\frac{1}{3}.
	\end{gathered}
\end{equation}
\end{itemize}
Note that if $f_2''=0$, we don't need the order condition $\bm{b}^\top (D^{(2)}\bm{c})^2=\frac{1}{3}$. Therefore, we have the following theorem which is useful for error analysis.
\begin{theorem}
	If the order conditions \eqref{prk_1stordcon}-\eqref{prk_3rdordcon} hold of a product-type IMEX-RK scheme \eqref{intrPRK1}, then this product-type IMEX-RK scheme achieves third-order accuracy and the local truncation error satisfies
	\begin{equation}
		\|\eta_n\|+\tau\|\bm{\eta}\|\leqslant C\tau^3,
	\end{equation}
	where $\eta_n$ and $\bm{\eta}$ are given by \eqref{localTE}, for sufficiently smooth exact solutions.
\end{theorem}

\subsection{Stability analysis}

In this section, we perform the absolute stability analysis of the PRK methods. Without loss of generality, we consider the ODE system
\begin{equation}\label{ODE_stability}
    u'=f(u/\varepsilon,u,h(u)),
\end{equation}
where the small parameter $\varepsilon$ represents the stiff term and the other terms are non-stiff. We analyze the stability of the PRK scheme 
\begin{equation}\label{Scheme_stability}
    \begin{aligned}
        &\bm{u} = u_0\bm{1} + \tau A f(D\bm{u}/\varepsilon,D^{(1)}\widehat{\bm{u}},D^{(2)}h(\widehat{\bm{u}})),\\
		&u^{n+1} = u_0 + \tau \bm{b}^\top f(D\bm{u}/\varepsilon,D^{(1)}\widehat{\bm{u}},D^{(2)}h(\widehat{\bm{u}})),
    \end{aligned}
\end{equation}
where the diagonal elements of $A,D$ satisfy $a_{ii},d_{ii}>0$.

To analyze the absolute stability of the PRK scheme \eqref{Scheme_stability}, it is necessary to derive a similar simplified equation, as traditional Runge-Kutta methods typically consider the ODE $u'=\lambda u$ with $\text{Re}(\lambda)<0$ for absolute stability analysis. Define $u = u^{(1)} - u^{(2)}$, where $u^{(1)}$ and $u^{(2)}$ satisfy the ODE system \eqref{ODE_stability}. Then we have
\begin{equation*}
\begin{aligned}
    u' =& f(u^{(1)}/\varepsilon,u^{(1)},h(u^{(1)}))-f(u^{(2)}/\varepsilon,u^{(2)},h(u^{(2)}))\\
    =&\partial_1 f(\xi_1/\varepsilon,\xi_2,h(\xi_3))\frac{u}{\varepsilon}+\partial_2 f(\xi_1/\varepsilon,\xi_2,h(\xi_3))u+\partial_3 f(\xi_1/\varepsilon,\xi_2,h(\xi_3))h'(\xi_3)u,
\end{aligned}
\end{equation*}
where $\xi_i$ is between $u^{(1)}$ and $u^{(2)}$ for $i=1,\,2,\,3$. If we assume $|\partial_i f|,|h'|\leqslant C$ and $C$ is independent of $\varepsilon$, we can analyze absolute stability properties of the PRK scheme \eqref{Scheme_stability} by considering the test ODE equation
\begin{equation}\label{ODE_A_stability}
    u'=\lambda_0 u + \lambda_1 u +\lambda_2 u,\quad t>0,
\end{equation}
with $\lambda_0,\lambda_1,\lambda_2\in \mathbb{C}$, where $\lambda_0 u$ corresponds to the stiff part and $\lambda_1u,\lambda_2u$ to the non-stiff parts. Applying the PRK scheme \eqref{Scheme_stability} to \eqref{ODE_A_stability}, we obtain the following additive IMEX-RK scheme
\begin{align}
        &u_0=u^n,\nonumber\\
        &u_i = u_0 +\tau \sum_{j=0}^i \hat{a}_{ij}\lambda_0 u_j + \tau \sum_{j=0}^i a^{(1)}_{ij}\lambda_1 u_j + \tau \sum_{j=0}^i a^{(2)}_{ij}\lambda_2 u_j,\quad i=0,\,\dots,\,s,\label{Astable_RK1}\\
        &u^{n+1} = u_0+\tau \sum_{i=0}^s \hat{b}_{i}\lambda_0 u_i +\tau \sum_{i=0}^s b^{(1)}_{i}\lambda_1 u_i+\tau \sum_{i=0}^s b^{(2)}_{i}\lambda_2 u_i,\label{Astable_RK2}
\end{align}
where the Butcher tableaux of the above IMEX-RK scheme satisfy
\begin{equation}\label{A_stable_tab}\renewcommand\arraystretch{1.2}
    \begin{aligned}
		\widehat{A}=(\hat{a}_{ij}):\,\begin{array}{c|cc}
			0 &0&\bm{0}^\top\\
			\bm{c} &\bm{0}&AD\\ \hline
			&0&\bm{b}^\top D
		\end{array},\,
		A^{(l)}=(a^{(l)}_{ij}):\,\begin{array}{c|cc}
			0 &\bm{0}^\top&0\\
			\bm{c} &AD^{(l)}&\bm{0}\\ \hline
			&\bm{b}^\top D^{(l)}&0
		\end{array},\;l=1,\,2,
	\end{aligned}
\end{equation}
with $\hat{\bm{b}}=(\hat{b}_i)=(0,\bm{b}^\top D)^\top$, $\bm{b}^{(l)}=(\hat{b}_i^{(l)})=(\bm{b}^\top D^{(l)},0)^\top\in \mathbb{R}^{(s+1)}$, and $\widehat{A},A^{(l)}\in \mathbb{R}^{(s+1)\times (s+1)}$ for $l=1,\,2$. 
Therefore, the discussion of absolute stability for PRK schemes is reduced to examining the absolute stability of additive IMEX-RK schemes. Although we can separately analyze the stability of each part of Runge-Kutta coefficients in \eqref{A_stable_tab}, this is not sufficient to guarantee desirable stability properties of the overall IMEX-RK scheme \cite{Hundsdorfer2003}.
The absolute stability of the additive IMEX-RK scheme with two Runge-Kutta tableaux was analyzed in \cite{Pareschi2000,IZZO201771}, which is similar to the PRK schemes. Therefore, we only give a brief absolute stability analysis of the PRK scheme \eqref{Scheme_stability}.

Denote $z_0=\tau \lambda_0,\,z_1=\tau \lambda_1,\,z_2=\tau \lambda_2$. Solving $u_i$ by \eqref{Astable_RK1} and substituting in \eqref{Astable_RK2}, we have 
\begin{equation*}
\begin{aligned}
    &u^{n+1}=\mathcal{R}(z_0,z_1,z_2)u^n\\
    =& \left[1+(z_0\hat{\bm{b}}+z_1\bm{b}^{(1)}+z_2\bm{b}^{(2)})^\top(I-z_0\widehat{A}-z_1 A^{(1)}-z_2 A^{(2)})^{-1}\bm{1}\right]u^n.
\end{aligned}
\end{equation*}
The function $\mathcal{R}(z_0,z_1,z_2)$ is the \textit{function of absolute stability}. The \textit{region of absolute stability} $\mathcal{A}$ associated with PRK scheme \eqref{Scheme_stability} is defined as
\begin{equation}
    \mathcal{A} = \Big\{(z_0,z_1,z_2)\in \mathbb{C}^3:|\mathcal{R}(z_0,z_1,z_2)|\leqslant 1\Big\}.
\end{equation}
We wish the PRK scheme is $A(\alpha)$- or $A$-stable with respect to the stiff and implicit part. Then we define the region of absolute stability of the explicit parts assuming that the implicit part is $A(\alpha)$-stable ($A$-stable if $\alpha=\frac{\pi}{2}$) by
\begin{equation}\label{S_alpha}
    \mathcal{S}_\alpha = \Big\{(z_1,z_2)\in \mathbb{C}^2:|\mathcal{R}(z_0,z_1,z_2)|\leqslant 1\text{ for }z_0\in \mathcal{A}_\alpha\Big\},
\end{equation}
where $\mathcal{A}_\alpha = \{z\in\mathbb{C}^{-}:|\text{Im}(z)|\leqslant \tan(\alpha)|\text{Re}(z)|\}$ for $\alpha\in (0,\frac{\pi}{2}]$ and $\mathbb{C}^{-}= \{z\in\mathbb{C}:\text{Re}(z)<0\}$. This implies that $\mathcal{A}_\alpha \times \mathcal{S}_\alpha$ is a subset of the region of absolute stability $\mathcal{A}$. Furthermore, if $\limsup_{z_0\rightarrow\infty}|\mathcal{R}(z_0,z_1,z_2)|<1$, we can prove that 
\begin{equation*}
    \mathcal{S}_\alpha = \bigcap_{y\in\mathbb{R}} \mathcal{S}_{y,\alpha}=\bigcap_{y\in\mathbb{R}} \Big\{(z_1,z_2)\in \mathbb{C}^2:|\mathcal{R}(z_0,z_1,z_2)|\leqslant 1\text{ for }z_0=-|y|/\tan(\alpha)+iy\Big\}
\end{equation*}
by using the maximum modulus principle. The proof is exactly the same as the proof in \cite{Pareschi2000}.

\begin{remark}

In practice, the assumption that $\lambda_1$ and $\lambda_2$ are non-stiff of \eqref{ODE_A_stability} does not always hold. For instance, consider
\begin{equation}\label{A_stab_pro}
u_t=f(u/\varepsilon,u,h(u))=\frac{u}{\varepsilon}g(u,h(u)).
\end{equation}
In such cases, $\lambda_1u$ and $\lambda_2u$ remain stiff. The time step restriction $\tau \leqslant C\varepsilon$ is necessary to ensure absolute stability. How to better analyze the absolute stability of \eqref{A_stab_pro} requires further exploration.

Nevertheless, the absolute stability discussed herein is not a strictly necessary condition for the overall stability of the numerical method. In other words, a numerical method
lacking absolute stability in this specific sense does not necessarily exhibit unstable numerical behavior in practice. By judiciously exploiting the underlying physical properties of the equations, we can design structure-preserving algorithms that circumvent overly restrictive time step constraints. Even though equation \eqref{Gradient Flow} is of the form \eqref{A_stab_pro}, we successfully develop an unconditionally stable numerical scheme utilizing a product-type IMEX-RK method.
\end{remark}

\section{Numerical scheme}\label{sec3}

    In this section, we introduce a PRK scheme for the constrained gradient flow equation \eqref{Gradient Flow}. Subsequently, we derive the conditions of preserving both energy dissipation and unit length for the proposed PRK scheme. We consider either time-independent Dirichlet boundary conditions or homogeneous Neumann boundary conditions. A second-order PRK tableau is given. We denote by $(\cdot,\cdot)$ the $L^2$-inner product and by $\|\cdot\|$ the $L^2$-norm.
    \subsection{Review of unit length and energy dissipative preserving methods}
    
    As discussed in the introduction, balancing the dual preservation of unit length and energy dissipation with computational efficiency is a primary challenge. Before presenting our proposed PRK framework, we briefly review two representative decoupled methods that maintain these dual structures: the first-order linear SIP1 scheme \cite{du2024} and the second-order weakly nonlinear LM2 scheme \cite{cheng2023length}. These two methods serve as the primary baselines for our subsequent efficiency and robustness comparisons.

    \subsubsection{Semi-implicit projection method} 
    The projection method is widely used to maintain the unit length. However, the energy is altered after projection. Based on the projection property \cite{Alouges1997,Alouges2008,du2024}, 
    \begin{equation}
        \label{Pe}
			\left|\nabla \left(\frac{\tilde{\bm{m}}}{|\tilde{\bm{m}}|}\right)\right|\leqslant \left|\nabla \tilde{\bm{m}}\right|,\quad \text{provided } |\tilde{\bm{m}}|\geqslant 1\, \text{a.e.},
	\end{equation}
    the projection is energy decreasing whenever the length before projection is no less than 1. The discrete version of \eqref{Pe} for finite element approximation  was proved by Bartels in \cite{bartels2005stability}.
    
    \textbf{First-order semi-implicit projection (SIP1) method \cite{Alouges2008,du2024}:}
    
    \textbf{Step 1} (Predictor): given $\bm{m}^n$, compute $\bm{\tilde{m}}^{n+1}$ from
\begin{equation*}
	\frac{\tilde{\bm{m}}^{n+1} - \bm{m}^n}{\tau}= P(\bm{m}^n)\left( \theta \Delta (\tilde{\bm{m}}^{n+1}-\bm{m}^n) +\Delta \bm{m}^n\right),
\end{equation*}
where $\frac{1}{2}\leqslant\theta\leqslant 1$.

    \textbf{Step 2} (Projection): perform projection by
\begin{equation*}
    \bm{m}^{n+1} = \frac{\tilde{\bm{m}}^{n+1}}{|\tilde{\bm{m}}^{n+1}|}.
\end{equation*}
Note that the SIP1 method obtains a predicted solution along a direction perpendicular to $\bm{m}^n$, that is, $|\tilde{\bm{m}}^{n+1}|\geqslant 1$. Therefore, the law of energy dissipation $\|\nabla \bm{m}^{n+1}\|\leqslant \|\nabla \tilde{\bm{m}}^{n+1}\|\leqslant \|\nabla \bm{m}^{n}\|$ is satisfied.

\begin{remark}
    The projection post-process is efficient and convenient. It is worth noting that projecting the numerical solution onto the unit sphere does not affect the accuracy, regardless of the order of the numerical algorithm. For example, consider a numerical method of $p$-th order accuracy. Its local truncation error satisfies $\bm{m}^n = \bm{m}(t_n) + O(\tau^{p+1})$. We then have
    $$|\bm{m}^n|^2 = |\bm{m}(t_n)|^2 + O(\tau^{p+1}) = 1 + O(\tau^{p+1}).$$
Therefore,
$$\frac{\bm{m}^n}{|\bm{m}^n|} = \frac{\bm{m}(t_n) + O(\tau^{p+1})}{\sqrt{1 + O(\tau^{p+1})}} = \bm{m}(t_n) + O(\tau^{p+1}).$$
Furthermore, \cite{du2024} proved that the error also decreases when $|\bm{m}^n| \geqslant 1$, i.e. $$\left| \frac{\bm{m}^n}{|\bm{m}^n|} - \bm{m}(t_n) \right| \leqslant |\bm{m}^n - \bm{m}(t_n)|.$$ Relevant literature that provides rigorous error estimates for this projection technique can refer to \cite{chen2021convergence, an2021optimal, gui2022convergence, Gui2024} and \cite{Li2026} for SIP1 method.
\end{remark}

    \subsubsection{Lagrange multiplier method} 
    The Lagrange multiplier approach introduces two Lagrange multipliers: $\lambda(\bm{x},t)$, a scalar, and $\bm{\xi}(t)$, a space-independent vector. With these, equation \eqref{Gradient Flow} can be rewritten as
    \begin{equation}
    \label{lm}
        \begin{aligned}
            &\bm{m}_t=\alpha\left(\Delta \bm{m}+\lambda(\bm{x},t)\bm{m}\right),\\
            &\bm{m} = \frac{\bm{m}+\bm{\xi}(t)}{|\bm{m}+\bm{\xi}(t)|},\\
            &\frac{\mathrm{d}}{\mathrm{d} t}E(\bm{m}) = -\alpha\int_{\Omega} |\bm{m}\times \Delta \bm{m}|^2\mathrm{d}\bm{x},
        \end{aligned}
    \end{equation}
    where $E(\bm{m}) =\int_{\Omega} \frac{1}{2}|\nabla \bm{m}|^2 \, \mathrm{d}\bm{x}$.
    For simplicity, we only introduce the case $\beta = 0$. The case for $\beta \neq 0$ can be found in \cite{cheng2023length}. It is evident that equations \eqref{lm} and \eqref{Gradient Flow} are equivalent, since $\bm{\xi}(t) = 0$ whenever $\bm{m}$ is not spatially constant, and $\lambda(\bm{x},t) = -\bm{m}\cdot\Delta \bm{m}$.
    
    \textbf{Second-order Lagrange multiplier (LM2) method \cite{cheng2023length}:}
    
    \textbf{Step 1} (Predictor): given $(\bm{m}^n,\lambda^n)$, compute $\tilde{\bm{m}}^{n+1}$ from
    \begin{align*}
        \frac{\tilde{\bm{m}}^{n+1}-\bm{m}^n}{\tau}=\alpha\left(\Delta \frac{\tilde{\bm{m}}^{n+1}+\tilde{\bm{m}}^n}{2}+\lambda^n \bm{m}^n\right),
    \end{align*}

    \textbf{Step 2} (Corrector): compute $(\hat{\bm{m}}^{n+1},\lambda^{n+1})$ from
    \begin{align*}
        &\frac{\hat{\bm{m}}^{n+1}-\tilde{\bm{m}}^{n+1}}{\tau}=\alpha\frac{\lambda^{n+1} \hat{\bm{m}}^{n+1}-\lambda^n \bm{m}^n}{2},\\
        &\left|\hat{\bm{m}}^{n+1}\right|=1,
    \end{align*}

    \textbf{Step 3} (Preserving energy dissipation): compute $(\bm{m}^{n+1},\bm{\xi}^{n+1})$ from
    \begin{align*}
        & \bm{m}^{n+1}=\frac{\hat{\bm{m}}^{n+1}+\bm{\xi}^{n+1}}{\left|\hat{\bm{m}}^{n+1}+\bm{\xi}^{n+1}\right|}, \\
& \frac{E^{n+1}-E^n}{\tau}=-\alpha\left\|\frac{\hat{\bm{m}}^{n+1}+\bm{m}^n}{2} \times \Delta \frac{\hat{\bm{m}}^{n+1}+\bm{m}^n}{2}\right\|^2,
    \end{align*}
where the energy approximation is $E^{n+1}=E[\bm{m}^{n+1}]=\int_{\Omega} \frac{1}{2}|\nabla \bm{m}^{n+1}|^2 \, \mathrm{d}\bm{x}$. 
Note that Step 3 requires solving a nonlinear algebraic equation. However, such a nonlinear algebraic equation cannot ensure solvability.

   
	\subsection{Product-type IMEX-RK schemes for \eqref{Gradient Flow}}
    Based on the idea of the SIP1 scheme, we extend the numerical method of preserving length and energy dissipation to higher orders. Consider the Butcher tableaux of $A = (a_{ij})_{s\times s}$ and $D = (d_{ij})_{s\times s}$ as follows:	
	\begin{equation}
		\begin{aligned}
		A:\,\begin{array}{c|cccc}
			c_1 &a_{11}&0&\cdots &0\\
			c_2 &a_{21}&a_{22}&\ddots &\vdots\\
			\vdots&\vdots&&\ddots&0\\
			c_s &a_{s1}&a_{s2}&\cdots &a_{ss}\\ \hline
			&b_1&b_2&\cdots &b_s
		\end{array},\,
		D:\,\begin{array}{c|cccc}
			1 &d_{11}&0&\cdots &0\\
			1 &d_{21}&d_{22}&\ddots &\vdots\\
			\vdots&\vdots&&\ddots&0\\
			1 &d_{s1}&d_{s2}&\cdots &d_{ss}\\ \hline
			&&& &
		\end{array},
	\end{aligned}
	\end{equation}
	where $c_i = \sum_{j=1}^i a_{ij}$ , $1 = \sum_{j=1}^i d_{ij}$ and $a_{ii},d_{ii}>0$, $i=1,\,2,\,\dots,\,s$.
	
	To determine $\bm{m}^{n+1}$ given $\bm{m}^n$, we design the PRK schemes as:
	\begin{equation}
		\label{PRK}
		\begin{aligned}
			&\bm{u}_0  =\bm{m}^n, \\
            &\bm{u}_i  = \bm{u}_0 +\tau \sum_{j=1}^i a_{ij} P(\bm{u}_{j-1})\sum_{k=1}^j d_{jk} \Delta \bm{u}_k,\quad i=1,\,2,\,\cdots,\,s,\\
			&\tilde{\bm{m}}^{n+1}  =\bm{u}_0 +\tau \sum_{j=1}^s b_{j} P(\bm{u}_{j-1})\sum_{k=1}^j d_{jk} \Delta \bm{u}_k,\\
			&\bm{m}^{n+1} = \frac{\tilde{\bm{m}}^{n+1}}{\left|\tilde{\bm{m}}^{n+1}\right|}.
		\end{aligned}
	\end{equation}
	It is easy to check that the projection $\frac{\tilde{\bm{m}}^{n+1}}{\left|\tilde{\bm{m}}^{n+1}\right|}$ preserves the Dirichlet boundary conditions and homogeneous Neumann boundary conditions.
\subsection{Energy dissipation and unit length preservation}
    In this subsection, we establish the energy dissipation conditions for the PRK scheme \eqref{PRK}. The key idea is to ensure $\| \nabla\tilde{\bm{m}}^{n+1} \|\leqslant\|  \nabla\bm{m}^n\|$ and $|\tilde{\bm{m}}^{n+1} |\geqslant| \bm{m}^n|$. Then, we use the projection property \eqref{Pe}
		to obtain the energy dissipation $\| \nabla\bm{m}^{n+1} \|\leqslant\|  \nabla\bm{m}^n\|$.
    

    \begin{theorem}\label{mainthm}
		Consider the product-type IMEX-RK scheme \eqref{PRK}. If $b_i\geqslant0$ and the matrices $Q = (q_{ij})_{s\times s}$ \eqref{Q_matrix} and $R = (r_{ij})_{s\times s}$ \eqref{R_matrix} are positive semi-definite, then the product-type IMEX-RK scheme  \eqref{PRK} preserves $|\bm{m}^{n+1}| = 1$ and the energy dissipation law
		\begin{equation}
			\int_\Omega \frac{1}{2} \left|\nabla \bm{m}^{n+1}\right| \mathrm{d}\bm{x}\leqslant \int_\Omega \frac{1}{2} \left|\nabla \bm{m}^{n}\right|\mathrm{d}\bm{x}.
		\end{equation}
	\end{theorem}
\begin{proof}
 For the convenience of reading, we denote $\tilde{\bm{u}}_i = \sum_{j=1}^i d_{i j} \bm{u}_j$ for $1\leqslant i\leqslant s$ and rewrite scheme \eqref{PRK} as
   	\begin{equation}
		\label{PRKm}
		\begin{aligned}
			&\bm{u}_0  =\bm{m}^n, \\
			&\bm{u}_i  =\bm{u}_0 +\tau\sum_{j=1}^i a_{ij}P(\bm{u}_{j-1}) \Delta \tilde{\bm{u}}_j,\quad i=1,\,2,\,\cdots,\,s,\\
			&\bm{m}^{n+1} = \frac{\tilde{\bm{m}}^{n+1}}{\left|\tilde{\bm{m}}^{n+1}\right|},\;\tilde{\bm{m}}^{n+1}  =\bm{u}_0 +\tau\sum_{j=1}^s b_{j}P(\bm{u}_{j-1}) \Delta \tilde{\bm{u}}_j.
		\end{aligned}
	\end{equation}
    
    Note that
    \begin{equation}\label{pu_rep}
        P(\bm{u}_{i-1})\Delta \tilde{\bm{u}}_i=\sum_{j=1}^i a^{ij}(\bm{u}_j-\bm{u}_0)/\tau,
    \end{equation}
    where $A$ is invertible and $A^{-1} = (a^{ij})_{s\times s}$. The boundary of $P(\bm{u}_{i-1})\Delta \tilde{\bm{u}}_i$ is the same as the boundary of $\sum_{j=1}^s a^{ij}(\bm{u}_j-\bm{u}_0)/\tau$. Therefore, we have integration by parts $\left(\nabla (P(\bm{u}_{i-1})\Delta \tilde{\bm{u}}_i),\bm{v}\right)  = -\left(P(\bm{u}_{i-1})\Delta \tilde{\bm{u}}_i,\nabla\cdot\bm{v}\right)$.
    Then it is easy to derive
	\begin{equation}\label{enmt}
	\begin{aligned}
		\left(\nabla \tilde{\bm{m}}^{n+1},\nabla\tilde{\bm{m}}^{n+1}\right) =& \left(\nabla\bm{u}_0 ,\nabla\bm{u}_0 \right) - 2\tau \sum_{i=1}^s b_i \left( P(\bm{u}_{i-1})\Delta \tilde{\bm{u}}_i,\Delta \bm{u}_0 \right) \\
		&+\tau^2 \sum_{i,j=1}^s b_i b_j \left(P(\bm{u}_{i-1})\Delta \tilde{\bm{u}}_i,(-\Delta) P(\bm{u}_{j-1})\Delta \tilde{\bm{u}}_j\right).
	\end{aligned}
		\end{equation} 
    On the other hand, we also have 
   	\begin{equation}\label{u0r}
	\bm{u}_0 =\sum_{j=1}^i d_{ij}\bm{u}_0= \tilde{\bm{u}}_i -\tau \sum_{j,k=1}^i d_{i k} a_{kj} P(\bm{u}_{j-1})\Delta \tilde{\bm{u}}_j,\quad i=1,\,2,\,\dots,\,s,
	\end{equation}
where we used the compatible condition $\sum_{j=1}^i d_{ij}=1$. 
	Substituting \eqref{u0r} into the second term on the right-hand side of \eqref{enmt} yields
	\begin{align*}
		\left(\nabla \tilde{\bm{m}}^{n+1},\nabla\tilde{\bm{m}}^{n+1}\right) =& \left(\nabla\bm{u}_0 ,\nabla\bm{u}_0 \right) - 2\tau \sum_{i=1}^s b_i\left(  P(\bm{u}_{i-1})\Delta \tilde{\bm{u}}_i,\Delta \tilde{\bm{u}}_i \right) \\
		&-\tau^2 \sum_{i,j=1}^s q_{ij} \left(P(\bm{u}_{i-1})\Delta \tilde{\bm{u}}_i,(-\Delta) P(\bm{u}_{j-1})\Delta \tilde{\bm{u}}_j\right),
	\end{align*}
	where
	\begin{equation}
		\label{Q_matrix}
        \begin{aligned}
            Q=&(q_{ij})_{s\times s} = \left(\sum_{k=1}^i b_i d_{i k} a_{kj}+\sum_{k=1}^j b_j d_{j k} a_{ki}-b_ib_j\right)_{s\times s}\\
            =&BDA + (BDA)^\top -\bm{b}\bm{b}^\top ,\quad B = \text{diag}(\bm{b}).
        \end{aligned}
	\end{equation}
    Therefore, when the matrix $Q = (q_{ij})_{s\times s}$ is positive semi-definite and $b_i\geqslant 0$, we have $\|\nabla\tilde{\bm{m}}^{n+1} \|\leqslant\| \nabla\bm{m}^n\|$.
	
	Similarly, for $|\tilde{\bm{m}}^{n+1}|$, we have
	\begin{align*}
		\tilde{\bm{m}}^{n+1}\cdot\tilde{\bm{m}}^{n+1} =& \bm{u}_0 \cdot\bm{u}_0  + 2\tau \sum_{i=1}^s b_i [ P(\bm{u}_{i-1})\Delta \tilde{\bm{u}}_i ] \cdot \bm{u}_0 \\
		&+\tau^2 \sum_{i,j=1}^s b_i b_j[ P(\bm{u}_{i-1})\Delta \tilde{\bm{u}}_i]\cdot [P(\bm{u}_{j-1})\Delta \tilde{\bm{u}}_j].
	\end{align*}
	Substituting $\bm{u}_0 = \bm{u}_i -\tau \sum_{j=1}^i  a_{ij} P(\bm{u}_{j-1})\Delta \tilde{\bm{u}}_j$
    into the above equation gives
	\begin{align*}
		\tilde{\bm{m}}^{n+1}\cdot\tilde{\bm{m}}^{n+1}
		=& \bm{u}_0 \cdot\bm{u}_0 + 2\tau \sum_{i=1}^s b_i [P(\bm{u}_{i-1})\Delta \tilde{\bm{u}}_i]\cdot \bm{u}_{i-1}\\
		&+\tau^2 \sum_{i,j=1}^s r_{ij}[ P(\bm{u}_{i-1})\Delta \tilde{\bm{u}}_i]\cdot[ P(\bm{u}_{j-1})\Delta \tilde{\bm{u}}_j]\\
        =&\bm{u}_0 \cdot\bm{u}_0 +\tau^2 \sum_{i,j=1}^s r_{ij} [P(\bm{u}_{i-1})\Delta \tilde{\bm{u}}_i]\cdot [P(\bm{u}_{j-1})\Delta \tilde{\bm{u}}_j],
	\end{align*}
	where
	\begin{equation}
		\label{R_matrix}
        \begin{aligned}
		R=&(r_{ij})_{s\times s} = \left(b_ib_j-b_i a_{i-1, j}-b_j a_{j-1,i}\right)_{s\times s}\\
        =&\bm{b}\bm{b}^\top  - BJA - (BJA)^\top 
        \end{aligned}
	\end{equation}
    with $a_{0j}=0$.
    Here we use $\bm{u}\cdot P(\bm{u})=\bm{0}$.
	Therefore, when the matrix $R = (r_{ij})_{s\times s}$ is positive semi-definite, we have $|\tilde{\bm{m}}^{n+1} |\geqslant| \bm{m}^n|$. According to \eqref{Pe}, we obtain the energy dissipation.
    \end{proof}
     
    \begin{remark}
        Similarly, we can prove the fully discrete unconditionally energy dissipation by using central finite difference method in Section \ref{sec4.1} or finite element method in \cite{bartels2005stability}, especially with the help of the discrete version of relation \eqref{pu_rep}. Actually, for the central finite difference method in Section \ref{sec4.1}, we can consider the identity relation
        \begin{equation}
            \bm{U}^i = \bm{U}^0 + \tau \sum_{j=1}^i a_{ij} \left(\sum_{k=1}^ja^{jk}\frac{\bm{U}^k-\bm{U}^0}{\tau}\right)
        \end{equation}
        and then derive the following equation
        \begin{equation}\label{eq3.13}
            \begin{aligned}
             &(\widetilde{\bm{M}}^{n+1},-D_h^{(3)}\widetilde{\bm{M}}^{n+1})_h \\
               =& (\bm{U}^0,-D_h^{(3)}\bm{U}^0)_h - 2\tau\sum_{i=1}^s b_i \left(\sum_{j=1}^i a^{ij}(\bm{U}^j-\bm{U}^0)/\tau,-D_h^{(3)}\sum_{j=1}^i d_{ij} \bm{U}_j\right)\\
               &-\tau^2\sum_{i,j=1}^sq_{ij}\left(\sum_{k=1}^i a^{ik}(\bm{U}^k-\bm{U}^0)/\tau,-D_h^{(3)}\sum_{k=1}^j a^{jk}(\bm{U}^k-\bm{U}^0)/\tau\right)_h
            \end{aligned}
        \end{equation}
        by using the same proof methodology for the semi-discrete case. Here, the discrete inner product $(\cdot,\cdot)_h$ is computed by the trapezoidal rule. Substituting the discrete version of relation \eqref{pu_rep} $$P(\bm{U}^{i-1})\sum_{j=1}^id_{ij}D^{(3)}_h\bm{U}^{j} = \sum_{j=1}^i a^{ij}(\bm{U}^j-\bm{U}^0)/\tau$$ into the second term of the right hand side of equation \eqref{eq3.13}, we obtain the fully discrete energy dissipation before projection $E_h(\widetilde{\bm{M}}^{n+1})\leqslant E_h(\bm{U}^0)$, where the discrete energy is defined by \eqref{eq4.1a}. Similarly, we can derive $\widetilde{\bm{M}}^{n+1}\cdot\widetilde{\bm{M}}^{n+1}\geqslant\bm{U}^0\cdot\bm{U}^0$. Then the fully discrete unconditionally energy dissipation follows from the discrete projection property \cite{du2024}:
        \begin{equation*}
        \begin{aligned}
            &((M^n_i)_{jk}-(M^n_i)_{(j-1)k})^2+((M^n_i)_{jk}-(M^n_i)_{j(k-1)})^2\\
            \leqslant&((\widetilde{M}^n_i)_{jk}-(\widetilde{M}^n_i)_{(j-1)k})^2+((\widetilde{M}^n_i)_{jk}-(\widetilde{M}^n_i)_{j(k-1)})^2].
        \end{aligned}
        \end{equation*}
        Although the nonlinear scheme in \cite{Fuwa2012} also satisfies the energy dissipation property, the CFL constraint essentially originates from the unique solvability of their nonlinear scheme. In contrast, the solvability of our linear scheme does not suffer any restriction on  the time step $\tau$ and the spatial step $h$. Consequently, there is no extra restriction to the time step $\tau$ and the spatial step $h$ required to guarantee stability, and our scheme is free from CFL constraints.
    \end{remark}

\subsection{A second-order PRK scheme}
	According to Theorem \ref{mainthm}, the Butcher tableau of $A$ and $D$ in a second-order unit length preserving and energy dissipative PRK scheme is given by
	\begin{equation}
    \renewcommand\arraystretch{1.5}\label{PRK2table}
		A = (a_{ij})_{2\times 2}: \begin{tabular}{c|cc}  $1$  & $1$ & 0  \\ $\frac{1}{2}$  & $0$ & $\frac{1}{2}$ \\  \hline & $\frac{1}{2}$ & $\frac{1}{2}$  \end{tabular}
		\qquad D = (d_{ij})_{2\times 2}: 
		\begin{tabular}{c|cccc}  $1$  & $1$ &0 \\ $1$  & $-1$ & $2$  \\ \hline &  &  \end{tabular}.
	\end{equation}
    Note that this Butcher tableau of $A$ and $D$ is not stiffly accurate, i.e., $b_j\neq a_{sj}$. Therefore, it is not L-stable. However, three-stage third-order PRK schemes do not exist. The existence of a high-stage third-order PRK scheme is open.
    
    \begin{property}
    Under constraints $b_i\geqslant 0$, $Q\geqslant 0$ and $R\geqslant 0$, do not exist two-stage second-order stiffly accurate product-type IMEX-RK schemes and three-stage third-order product-type IMEX-RK schemes.
    \end{property}
    \begin{proof}
        Observe that $\bm{1}^\top Q\bm{1}= 0 $ and $\bm{1}^\top R\bm{1}= 0$ by the second-order condition. Since $Q$ and $R$ are positive semi-definite, we find $\bm{1}$ is an eigenvector with eigenvalue zero of the matrices $Q$ and $R$, i.e., $Q\bm{1}=R\bm{1}= 0\cdot \bm{1}=\bm{0}$.

        For two-stage second-order PRK schemes, we obtain
        $b_1=b_2=\frac{1}{2},\; a_{11}=1$
        by $R\bm{1}=\bm{0}$ and the order conditions. If the IMEX-RK scheme is stiffly accurate, then $a_{21}=a_{22}=\frac{1}{2}$. Taking it into $Q\bm{1}=\bm{0}$, we have $d_{21}=-2$ and $d_{22}=2$ that violate $d_{21}+d_{22}=1$. Therefore, there is no two-stage second-order stiffly accurate PRK scheme.

        For three-stage third-order PRK schemes, we have
        $c_2 = 1,\; b_2c_1+b_3c_2 = \frac{1}{2}$
        by $R\bm{1}=\bm{0}$ and the second-order condition. Combining with the third-order condition
        $b_3c_2a_{11} = \frac{1}{6},\,b_2c_1^2+b_3c_2^2=\frac{1}{3},$
        we derive $ 12b_3^2-6b_3+1=0$.
        However, this quadratic equation does not have real solutions. Thus, there is no three-stage third-order PRK scheme.
    \end{proof}
    \begin{remark}
        For general nonlinear energy
        \begin{equation}
            E[\bm{m}(\bm{x})] = \int_{\Omega} \frac{1}{2}|\nabla \bm{m}|^2\, \mathrm{d}\bm{x}+E_1[\bm{m}(\bm{x})] ,\quad  E_1[\bm{m}(\bm{x})] =\int_{\Omega}F(\bm{m})\, \mathrm{d}\bm{x}
        \end{equation}
        the evolution equation is given by
        \begin{equation}
            \begin{aligned}
                 &\bm{m}_t = - P(\bm{m})\bm{\mu},\\
                 &\bm{\mu} = \frac{\delta E}{\delta \bm{m}} = -\Delta \bm{m} + F'(\bm{m}), \\
                 &|\bm{m}(\bm{x},0)|=1.
            \end{aligned}
        \end{equation}
        We can design a numerical scheme that preserves both the unit length constraint and a modified energy dissipation by using SAV \cite{shen2018scalar,shen2019new}:
        \begin{equation}
		\label{PRKsav}
		\begin{aligned}
			&\bm{u}_0  =\bm{m}^n,\;r_0=r^n \\         
            &\bm{u}_i  = \bm{u}_0 +\tau \sum_{j=1}^i a_{ij}\bm{\mu_j},\;r_i = r_0+\frac{\tau}{2}\sum_{j=1}^i a_{ij}\left(\frac{F'(\bm{u}_{j-1})}{\sqrt{E_1[\bm{u}_{j-1}]}},\bm{\mu_j}\right) ,\quad i=1,\,\dots,\,s,\\
            &\bm{\mu}_i = P(\bm{u}_{i-1})\left[\sum_{k=1}^i d_{ik}\Delta \bm{u}_k-\frac{F'(\bm{u}_{i-1})}{\sqrt{E_1[\bm{u}_{i-1}]}}\sum_{k=1}^i d_{ik}r_k\right],\quad i=1,\,\dots,\,s,\\
			&\bm{m}^{n+1} = \frac{\tilde{\bm{m}}^{n+1}}{\left|\tilde{\bm{m}}^{n+1}\right|},\,\tilde{\bm{m}}^{n+1}  =\bm{u}_0 +\tau \sum_{j=1}^s b_{j} \bm{\mu}_j,\,r^{n+1} = r_0+\frac{\tau}{2}\sum_{j=1}^s b_j \left(\frac{F'(\bm{u}_{j-1})}{\sqrt{E_1[\bm{u}_{j-1}]}},\bm{\mu_j}\right).
		\end{aligned}
	\end{equation}
    The proof is analogous to Theorem \ref{mainthm}. Denote $\tilde{\bm{u}}_i = \sum_{j=1}^i d_{i j} \bm{u}_j$. We can derive $|\tilde{\bm{m}}^{n+1} |\geqslant| \bm{m}^n|$ by 
    \begin{align}\label{eq3.17}
		\tilde{\bm{m}}^{n+1}\cdot\tilde{\bm{m}}^{n+1}=
		\bm{u}_0 \cdot\bm{u}_0 +\tau^2 \sum_{i,j=1}^s r_{ij} \bm{\mu}_i\cdot \bm{\mu}_j,
	\end{align}
    and modified energy dissipation
     $\frac{1}{2}\| \nabla\tilde{\bm{m}}^{n+1}\|^2 +(r^{n+1})^2 \leqslant\|  \frac{1}{2}\nabla\bm{m}^n\|^2 +(r^{n})^2$ by deriving the following relations
     \begin{equation*}
	   \begin{aligned}
		\left(\nabla \tilde{\bm{m}}^{n+1},\nabla\tilde{\bm{m}}^{n+1}\right) =& \left(\nabla\bm{u}_0 ,\nabla\bm{u}_0 \right) - 2\tau \sum_{i=1}^s b_i\left( \Delta \tilde{\bm{u}}_i,\bm{\mu}_i \right) -\tau^2 \sum_{i,j=1}^s q_{ij} \left(\bm{\mu}_i,(-\Delta) \bm{\mu}_j\right),
    \end{aligned}
	  \end{equation*}
    \begin{equation*}
	   \begin{aligned}
        \left(r^{n+1},r^{n+1}\right) =& \left(r_0 ,r_0 \right) + \tau \sum_{i=1}^s b_i\left(\frac{F'(\bm{u}_{i-1})}{\sqrt{E_1[\bm{u}_{i-1}]}}\sum_{k=1}^i d_{ik}r_k,\bm{\mu}_i \right) \\
		&-\frac{\tau^2}{4} \sum_{i,j=1}^s q_{ij} \left(\frac{F'(\bm{u}_{i-1})}{\sqrt{E_1[\bm{u}_{i-1}]}},\bm{\mu_i}\right)\left(\frac{F'(\bm{u}_{j-1})}{\sqrt{E_1[\bm{u}_{j-1}]}},\bm{\mu_j}\right),
	   \end{aligned}
	  \end{equation*}
    provided the same conditions in Theorem \ref{mainthm}.

    Furthermore, if the nonlinear energy is quadratic $F(\bm{m}) = \frac{1}{2}\bm{m}\cdot \tilde{\mathcal{L}}\bm{m} + \bm{v}\cdot \bm{m}$, satisfying
    \begin{equation}\label{quad en}
        E\left[\frac{\bm{m}}{|\bm{m}|}\right]\leqslant E[\bm{m}],\quad\text{for }|\bm{m}|\geqslant 1,
    \end{equation}
    we can design a numerical scheme that preserves both the unit length constraint and the original energy dissipation:
    \begin{equation}
        \begin{aligned}
            \label{PRKori}
            \begin{aligned}
			&\bm{u}_0  =\bm{m}^n\\         
            &\bm{u}_i  = \bm{u}_0 +\tau \sum_{j=1}^i a_{ij}\bm{\mu_j} ,\quad i=1,\,\dots,\,s,\\
            &\bm{\mu}_i = P(\bm{u}_{i-1})\left[\sum_{k=1}^i d_{ik}(\Delta \bm{u}_k-\gamma  \mathcal{L}\bm{u}_k)+\gamma \mathcal{L} \bm{u}_{i-1}-F'(\bm{u}_{i-1})\right],\quad i=1,\,\dots,\,s,\\
			&\bm{m}^{n+1} = \frac{\tilde{\bm{m}}^{n+1}}{\left|\tilde{\bm{m}}^{n+1}\right|},\;\tilde{\bm{m}}^{n+1}  =\bm{u}_0 +\tau \sum_{j=1}^s b_{j} \bm{\mu}_j,
		\end{aligned}
        \end{aligned}
    \end{equation}
    where $\gamma$ is a stabilization constant and $\mathcal{L}\succeq \tilde{\mathcal{L}}$ is a suitable linear operator.
    For example, the Landau-Lifshitz-Gilbert equation of a uniaxial material \cite{xie2020} satisfies the energy relation \eqref{quad en}.
    Actually, we can derive \eqref{eq3.17} and the following relations by a similar argument as in Theorem \ref{mainthm}
    \begin{equation*}
        \begin{aligned}
            \left( \tilde{\bm{m}}^{n+1},\mathcal{L}_1\tilde{\bm{m}}^{n+1}\right) =& \left(\bm{u}_0 ,\mathcal{L}_1\bm{u}_0 \right) + 2\tau \sum_{i=1}^s b_i\left(  \mathcal{L}_1 \tilde{\bm{u}}_i,\bm{\mu}_i \right)-\tau^2 \sum_{i,j=1}^s q_{ij} \left(\bm{\mu}_i,\mathcal{L}_1\bm{\mu}_j\right),\\
        \left( \tilde{\bm{m}}^{n+1},\mathcal{L}_2\tilde{\bm{m}}^{n+1}\right) =& \left(\bm{u}_0 ,\mathcal{L}_2\bm{u}_0 \right) + 2\tau \sum_{i=1}^s b_i\left( \mathcal{L}_2 \bm{u}_{i-1},\bm{\mu}_i \right) +\tau^2 \sum_{i,j=1}^s r_{ij} \left(\bm{\mu}_i,\mathcal{L}_2\bm{\mu}_j\right),\\
        \left( \bm{v},\tilde{\bm{m}}^{n+1}\right) =& \left(\bm{v},\bm{u}_0 \right) + \tau \sum_{i=1}^s b_i\left( \bm{v},\bm{\mu}_i \right),
        \end{aligned}
    \end{equation*}
    where $\mathcal{L}_1$ and $\mathcal{L}_2$ are linear operators, and denote $\tilde{\bm{u}}_i = \sum_{j=1}^i d_{i j} \bm{u}_j$. Substituting $\mathcal{L}_1=\frac{1}{2}(-\Delta +\gamma\mathcal{L})$ and $\mathcal{L}_2=\frac{1}{2}(-\gamma\mathcal{L} + \tilde{\mathcal{L}})$ and summing the above relations, we have the original energy dissipation, provided the same conditions in Theorem \ref{mainthm} and a suitable stabilization constant such that $\gamma(Q+R)\succeq R$.

    For the above product-type IMEX-RK schemes, the conditions for the preservation of both the unit length constraint and energy dissipation are the same as in Theorem \ref{mainthm}. Hence, we can use the Butcher tableau \eqref{PRK2table} to achieve a second-order scheme.
    \end{remark}
           
\section{Numerical experiments}\label{sec4}
	\label{numerical_exp}
	In this section, we present some numerical examples for the Landau-Lifshitz-Gilbert equation and the one-constant Oseen-Frank model. 
    Non-homogeneous Dirichlet or homogeneous Neumann boundary conditions are imposed as appropriate for each problem.

    \subsection{Implementation of PRK2 scheme}\label{sec4.1}
    In numerical experiments, we use the central finite difference method on a uniform mesh for spatial discretizations. Without loss of generality, we only present the two-dimensional case with the homogeneous Neumann boundary conditions. The other cases are similar. Let $\Omega = (0,L)^2$ and $h = L/K$. Denote $\bm{M}^n_{ij} = ((M_1^n)_{ij},(M_2^n)_{ij},(M_3^n)_{ij})^\top$ as the numerical solution at $\bm{x} = \bm{x}_{ij} = (ih,jh)^\top$ and $t=t_n=n\tau$.  We write $\bm{M}^n$ as a vector $\bm{M}^n = (\bm{M}^n_1,\bm{M}^n_2,\bm{M}^n_3)^\top$, where
    \begin{equation*}
    \begin{aligned}
        \bm{M}^n_1 = &((M_1^n)_{00},\dots,(M_1^n)_{K0};\dots;(M_1^n)_{K0},\dots,(M_1^n)_{KK})^\top;\\
        \bm{M}^n_2=&((M_2^n)_{00},\dots,(M_2^n)_{K0};\dots;(M_2^n)_{K0},\dots,(M_2^n)_{KK})^\top;\\
        \bm{M}^n_3=&((M_3^n)_{00},\dots,(M_3^n)_{K0};\dots;(M_3^n)_{K0},\dots,(M_3^n)_{KK})^\top.
    \end{aligned}
    \end{equation*}
    The discrete Laplacian on 3-dimensional vector field is defined as $D_h^{(3)}$, where  
    \begin{equation*}
        D_h^{(3)} = \begin{pmatrix}
            D_h&&\\
            &D_h&\\
            &&D_h
        \end{pmatrix},\quad
        G_h = \begin{pmatrix}
            -2&2&&&\\
            1&-2&1&&\\
            &\ddots&\ddots&\ddots&\\
            &&1&-2&1\\
            &&&2&-2
        \end{pmatrix},
    \end{equation*}
    and $D_h = I \otimes G_h + G_h \otimes I\in \mathbb{R}^{K^2\times K^2}$.
    We define the projection operator
    $$P(\bm{M}^n) = \alpha P_1(\widehat{\bm{M}}^{n})+\beta  \begin{pmatrix}
            \bm{0} &-\text{diag}(\bm{M}^n_3)&\text{diag}(\bm{M}^n_2)\\
            \text{diag}(\bm{M}^n_3)&\bm{0}&-\text{diag}(\bm{M}^n_1)\\
            -\text{diag}(\bm{M}^n_2)&\text{diag}(\bm{M}^n_1)&\bm{0}
        \end{pmatrix},$$
        where $\widehat{\bm{M}}^{n}_{ij} = \frac{\bm{M}^{n}_{ij}}{\sqrt{(M^{n}_1)_{ij}^2+(M^{n}_2)_{ij}^2+(M^{n}_3)_{ij}^2}}$ and
    \begin{equation*}
        P_1(\widehat{\bm{M}}^{n}) = I-\begin{pmatrix}
            \text{diag}(\widehat{\bm{M}}^n_1\circ \widehat{\bm{M}}^n_1)&\text{diag}(\widehat{\bm{M}}^n_1\circ \widehat{\bm{M}}^n_2)&\text{diag}(\widehat{\bm{M}}^n_1\circ \widehat{\bm{M}}^n_3)\\
            \text{diag}(\widehat{\bm{M}}^n_2\circ \widehat{\bm{M}}^n_1)&\text{diag}(\widehat{\bm{M}}^n_2\circ \widehat{\bm{M}}^n_2)&\text{diag}(\widehat{\bm{M}}^n_2\circ \widehat{\bm{M}}^n_3)\\
            \text{diag}(\widehat{\bm{M}}^n_3\circ \widehat{\bm{M}}^n_1)&\text{diag}(\widehat{\bm{M}}^n_3\circ \widehat{\bm{M}}^n_2)&\text{diag}(\widehat{\bm{M}}^n_3\circ \widehat{\bm{M}}^n_3)
        \end{pmatrix}.
    \end{equation*}
    The discrete energy is defined by 
    \begin{equation}\label{eq4.1a}
        E_h(\bm{M}) = (\bm{M},-D_h^{(3)}\bm{M})_h,
    \end{equation} 
    where the discrete inner product $(\cdot,\cdot)_h$ is computed by the trapezoidal rule. Therefore, we have
        $$(\bm{M},-D_h^{(3)}\bm{M})_h=\sum_{l=1}^3\sum_{j,k=0}^{K-1}\left[((M_l)_{jk}-(M_l)_{j(k-1)})^2+((M_l)_{jk}-(M_l)_{(j-1)k})^2\right].$$
    Then we can write the fully discrete PRK2 scheme as Algorithm \ref{alg:PRK}.
    
\begin{algorithm}[ht]
\caption{Fully discrete PRK2 Scheme}\label{alg:PRK}
\begin{algorithmic}[1]
\State \textbf{Given:} $\bm{M}^n$, time step $\tau$, coefficients $a_{ij}$, $b_j$, $d_{jk}$ for $i,j,k = 1,\dots,s$
\State \textbf{Initialize:} $\bm{U}^0 = \bm{M}^n$ 
\For{$i = 1$ to $2$}
    \State Solve the linear system for $\bm{U}^i$:
    \begin{align*}
        &\bm{\mu}_i = \left\{\begin{aligned}
        &\bm{0},&\text{if }i=1\\
            &\sum_{j=1}^{i-1} a_{ij} P(\bm{U}^{j-1}) \sum_{k=1}^{j} d_{jk} D_h^{(3)} \bm{U}^k
        + a_{ii} P(\bm{U}^{i-1}) \sum_{k=1}^{i-1} d_{ik} D_h^{(3)} \bm{U}^k,\;&\text{if }i>1
        \end{aligned}\right.   \\
        &(I - \tau a_{ii} d_{ii}  P(\bm{U}^{i-1}) D_h^{(3)}) \bm{U}^i 
        = \bm{U}^0 + \tau \bm{\mu}_i
    \end{align*}
    \Comment{Requires solving a linear system with sparse operator $I - \tau \, a_{ii} d_{ii} P(\bm{U}^{i-1}) D_h^{(3)}$}
\EndFor
\State Compute intermediate solution:
\[
\widetilde{\bm{M}}^{n+1} = \bm{U}^0 + \tau \sum_{j=1}^{2} b_{j} P(\bm{U}^{j-1}) \sum_{k=1}^{j} d_{jk} D_h^{(3)} \bm{U}^k
\]
\State Normalize to obtain the final solution:
\[
\bm{M}^{n+1}_{ij} = \frac{\widetilde{\bm{M}}^{n+1}_{ij}}{\sqrt{(\widetilde{M}^{n+1}_1)_{ij}^2+(\widetilde{M}^{n+1}_2)_{ij}^2+(\widetilde{M}^{n+1}_3)_{ij}^2}}.
\]
\Return $\bm{M}^{n+1}$
\end{algorithmic}
\end{algorithm}

Each step of PRK2 scheme consists of two stages, and each stage requires solving a linear system. In total, two linear systems need to be solved in our PRK2 scheme. It is worth noting that the matrices $(I - \tau \, a_{ii} \, d_{ii} \, P(\bm{U}^{i-1})D^{(3)}_h)$ are sparse. Each of them contains approximately $O(45K^2)$ non-zero elements distributed across 15 diagonals. Therefore, each linear system in PRK2 scheme can be solved efficiently. The obtained linear systems are solved by the built-in MATLAB solver.

In numerical experiments, we also compare the numerical performance of a PRK2 scheme in the form \eqref{intrPRK2} without provable energy dissipation
    \begin{equation}\label{other-PRK2}
	\begin{aligned}
		&\bm{u}_0  =\bm{m}^n, \\
		&\bm{u}_i  = \bm{u}_0+\tau\sum_{j=1}^i \hat{a}_{i j}\Big( \sum_{k=1}^j g_{jk}P(\bm{u}_{k-1}) \Big) \Delta \bm{u}_j, \quad \text{for }i=1,\,\dots,\, s, \\
		&\bm{m}^{n+1} = \frac{\tilde{\bm{m}}^{n+1}}{\left|\tilde{\bm{m}}^{n+1}\right|},\;\tilde{\bm{m}}^{n+1}  =\bm{u}_0 +\tau\sum_{j=1}^s \hat{b}_{ j}\Big( \sum_{k=1}^j g_{jk}P(\bm{u}_{k-1}) \Big) \Delta \bm{u}_j
	\end{aligned}
\end{equation}
with the Butcher tableau
\begin{equation}\renewcommand\arraystretch{1.2}\label{other-PRK2_table}
    \begin{aligned}
		\widehat{A}=(\hat{a}_{ij}):\,\begin{array}{c|c}
			\bm{c} &AD\\ \hline
			&\bm{b}^\top D
		\end{array},\quad
		G=(g_{ij}):\,\begin{array}{c|c}
			\bm{1} &D^{-1}\\ \hline
			&
		\end{array},
	\end{aligned}
\end{equation}
where the PRK coefficient matrices $A$ and $D$ are defined by \eqref{PRK2table}. Therefore, the absolute stability region of PRK scheme \eqref{other-PRK2} with Butcher tableau \eqref{other-PRK2_table} is the same as PRK sheme \eqref{PRK} with Butcher tableau \eqref{PRK2table}. We plot the absolute stability region $\mathcal{S}_{\frac{\alpha}{2}}$ in Fig.~\ref{fig:stabilityregion}. 
To make a distinction, we denote PRK scheme \eqref{PRK} as PRK2 and PRK scheme \eqref{other-PRK2} as scheme \eqref{other-PRK2}.

	\begin{figure}[!htbp]
	\centering	
    {\includegraphics[trim=0.3cm 0cm 0.5cm 0cm, clip,width=3in]{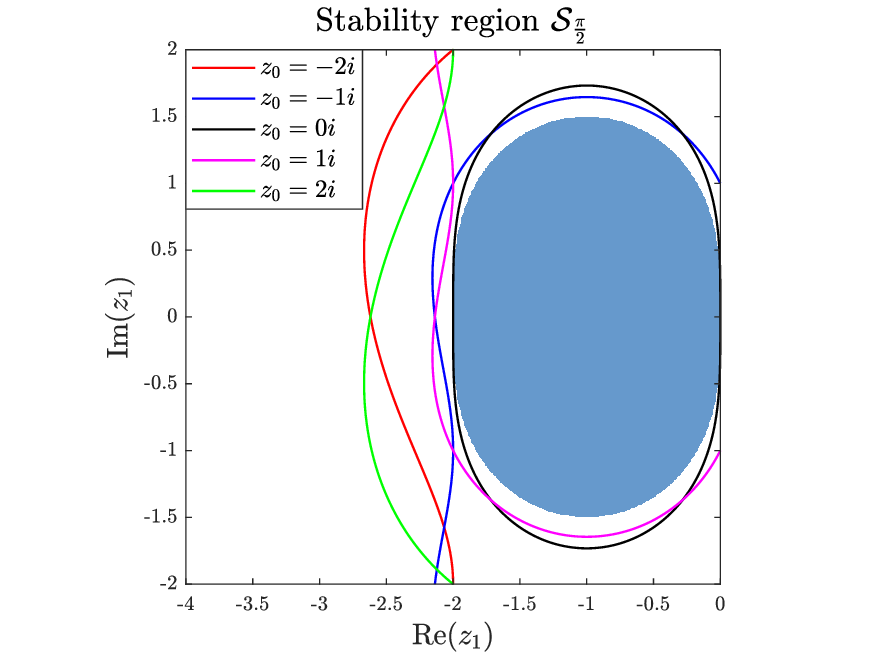}}
	\caption{Absolute stability region $\mathcal{S}_{\frac{\pi}{2}}$ of the Butcher tableau \eqref{PRK2table} for PRK sheme \eqref{PRK} and the Butcher tableau \eqref{other-PRK2_table} for PRK scheme \eqref{other-PRK2}.\label{fig:stabilityregion} }	
	\end{figure}

\subsection{Convergence Rate Test}
\begin{example}\label{ex_1}
{\rm
To investigate the convergence rates, we consider the gradient flow \eqref{Gradient Flow} in the computational domain $\Omega = (-\frac{1}{2}, \frac{1}{2})^2$ with parameters $\alpha = \beta = 1$. The initial condition is prescribed as follows:
\begin{equation*}
\begin{aligned}
    m_1(x, y, 0) &= 0.3 \sin(\pi x) \sin(\pi y), \\
    m_2(x, y, 0) &= 0.3 \sin(3\pi x) \sin(\pi y), \\
    m_3(x, y, 0) &= 1.0 + 0.2 \cos(2\pi x) \cos(2\pi y).
\end{aligned}
\end{equation*}
}
\end{example}
For spatial discretization, a uniform mesh with $h = 1/64$ is employed. Since the exact solution is unknown, a reference solution is generated using a fourth-order BDF method with a sufficiently small time step $\tau = 10^{-6}$. 

Table~\ref{tab:convergence} reports the $L^2$ errors and corresponding
convergence rates for the SIP1 scheme \cite{du2024}, the LM2 scheme
\cite{cheng2023length}, the proposed PRK2 scheme \eqref{PRK}, and the
scheme \eqref{other-PRK2}. The results confirm that SIP1
achieves first-order accuracy, while LM2, PRK2, and scheme
\eqref{other-PRK2} all exhibit consistent second-order convergence.

Despite sharing the same formal order, a clear difference in accuracy
levels is observed: both PRK-based schemes yield $L^2$ errors that are
approximately one order of magnitude smaller than those of LM2 at
identical time steps. Moreover, although a rigorous proof of energy
dissipation for scheme \eqref{other-PRK2} is left for future work, its
numerical accuracy is nearly indistinguishable from that of the proven
PRK2 scheme, highlighting the intrinsic accuracy and robustness of the
product-type IMEX Runge--Kutta framework.
\begin{table}[htbp]
\centering
\caption{Temporal convergence test: $L^2$ errors for
$\bm m=(m_1,m_2,m_3)^\top$ at $t=0.01024$.}
\label{tab:convergence}
\setlength{\tabcolsep}{5pt} 
\begin{tabularx}{\textwidth}{c
  r c
  r c
  r c
  r c}
\toprule
& \multicolumn{2}{c}{SIP1}
& \multicolumn{2}{c}{PRK2}
& \multicolumn{2}{c}{scheme \eqref{other-PRK2}}
& \multicolumn{2}{c}{LM2} \\
\cmidrule(lr){2-3}\cmidrule(lr){4-5}\cmidrule(lr){6-7}\cmidrule(lr){8-9}
$\tau$
& Error & Order
& Error & Order
& Error & Order
& Error & Order \\
\midrule
$3.2\mathrm{e}{-4}$
& $1.77\mathrm{e}{-3}$ & --
& $2.67\mathrm{e}{-5}$ & --
& $2.67\mathrm{e}{-5}$ & --
& $1.05\mathrm{e}{-3}$ & -- \\

$3.2\mathrm{e}{-4}/2$
& $8.94\mathrm{e}{-4}$ & 0.99
& $6.97\mathrm{e}{-6}$ & 1.95
& $6.93\mathrm{e}{-6}$ & 1.94
& $7.37\mathrm{e}{-5}$ & 3.83 \\

$3.2\mathrm{e}{-4}/2^{2}$
& $4.49\mathrm{e}{-4}$ & 0.99
& $1.78\mathrm{e}{-6}$ & 1.97
& $1.77\mathrm{e}{-6}$ & 1.97
& $1.87\mathrm{e}{-5}$ & 1.98 \\

$3.2\mathrm{e}{-4}/2^{3}$
& $2.25\mathrm{e}{-4}$ & 1.00
& $4.50\mathrm{e}{-7}$ & 1.98
& $4.47\mathrm{e}{-7}$ & 1.98
& $4.69\mathrm{e}{-6}$ & 1.99 \\

$3.2\mathrm{e}{-4}/2^{4}$
& $1.13\mathrm{e}{-4}$ & 1.00
& $1.14\mathrm{e}{-7}$ & 1.98
& $1.13\mathrm{e}{-7}$ & 1.98
& $1.18\mathrm{e}{-6}$ & 2.00 \\

$3.2\mathrm{e}{-4}/2^{5}$
& $5.64\mathrm{e}{-5}$ & 1.00
& $2.96\mathrm{e}{-8}$ & 1.94
& $2.95\mathrm{e}{-8}$ & 1.94
& $2.94\mathrm{e}{-7}$ & 2.00 \\
\bottomrule
\end{tabularx}
\end{table}

	\subsection{The Landau-Lifshitz-Gilbert equation}
    Consider the Landau-Lifshitz equation in the form \cite{bartels2006convergence, du2024}:
	\begin{equation}\label{LLG}
		\begin{aligned}
				&\bm{m}_t = \alpha\left(I - \bm{m}\bm{m}^\top /|\bm{m}^2|\right)\Delta\bm{m} + \beta \bm{m}\times\Delta\bm{m},\\
				&|\bm{m}|=1,
		\end{aligned}
	\end{equation}
	with homogeneous Neumann boundary conditions.

	\begin{example}\label{ex_2}{\rm
	We set $\Omega = [-\frac{1}{2}, \frac{1}{2}]^2$ and test the benchmark problem from \cite{an2021optimal} with the following smooth initial condition:
	$$
	\bm{m}(\bm{x},0)= \begin{cases}(0,0,-1)^\top  & \text { if }|\bm{x}| \geqslant 1 / 2, \\ \left(2 x_1 A,\, 2x_2A,\, A^2-|\bm{x}|^2\right)^\top  /\left(A^2+|\bm{x}|^2\right) & \text { if }|\bm{x}| < 1 / 2,\end{cases}
	$$
	where $A=(1-2|\bm{x}|)^4$.
    }
	\end{example}	

	We apply the PRK2 scheme \eqref{PRK} with parameters \( \alpha =  \beta = 1 \), \( h = 1/24 \), and \( \tau = 10^{-4} \) to compute numerical approximations of equation \eqref{LLG}. Figure~\ref{fig:LLG-top view} shows the numerical solution \( \bm{m} \) computed using our scheme on the plane \( \{(x, y, 0) : x, y \in \mathbb{R}\} \) at various time points. Rapid variation in \( \bm{m} \) is observed near the origin, suggesting a rapid growth of \( \nabla \bm{m} \), which may lead to divergence.
	
	At approximately \( t \approx 0.05 \), the spin \( \bm{m} \) at the origin transitions from \( (0, 0, 1)^\top  \) to \( (0, 0, -1)^\top  \), as shown in Figure~\ref{fig:LLG-front view}. This transition aligns with the blow up phenomenon described in \cite{bartels2008numerical}.
	
	
	\begin{figure}[!htbp]
		\centering	
		\subfloat[ Initial condition.]{\includegraphics[trim=1cm 0cm 1.5cm 0.5cm, clip,width=0.32\linewidth]{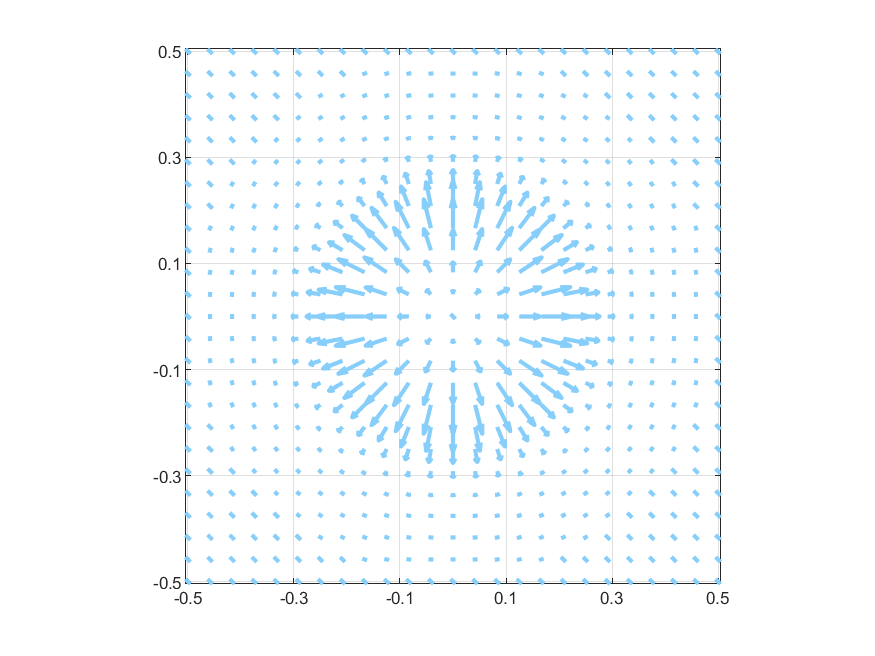}}
		\subfloat[$ t=0.02 $.]{\includegraphics[trim=1cm 0cm 1.5cm 0.5cm, clip,width=0.32\linewidth]{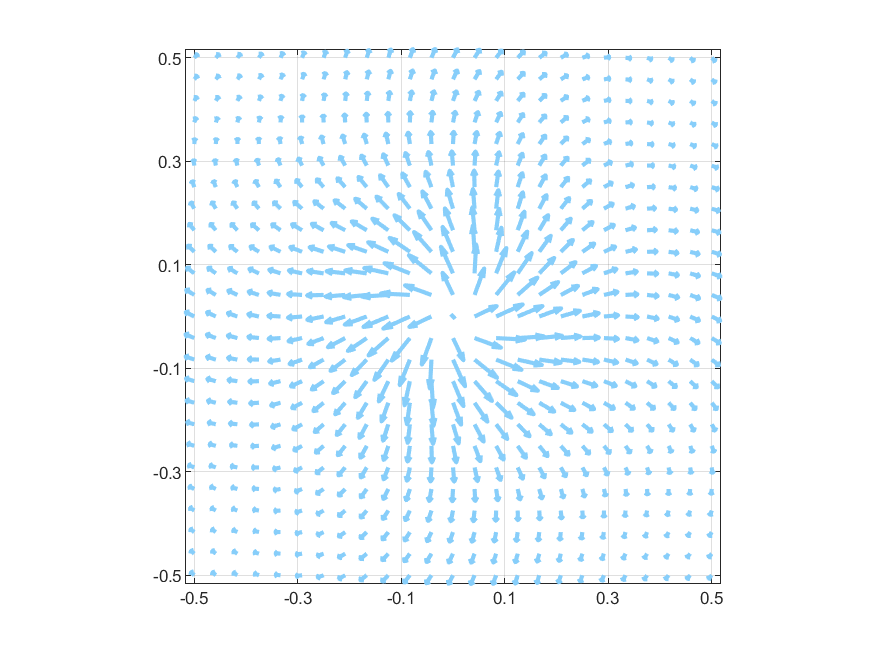}}
		\subfloat[$ t=0.04 $.]{\includegraphics[trim=1cm 0cm 1.5cm 0.5cm, clip,width=0.32\linewidth]{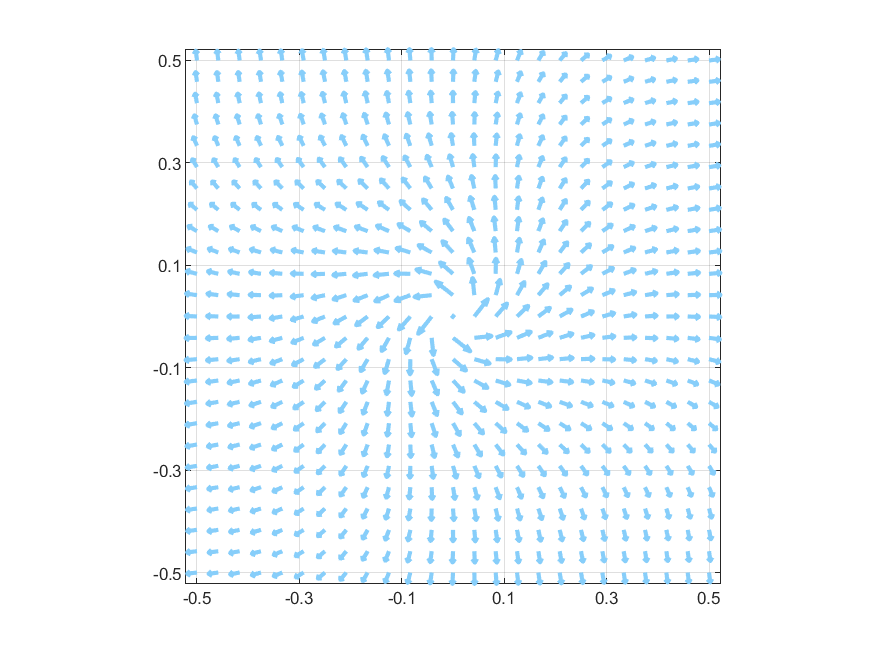}}\\
		\subfloat[$ t=0.048 $.]{\includegraphics[trim=1cm 0cm 1.5cm 0.5cm, clip,width=0.32\linewidth]{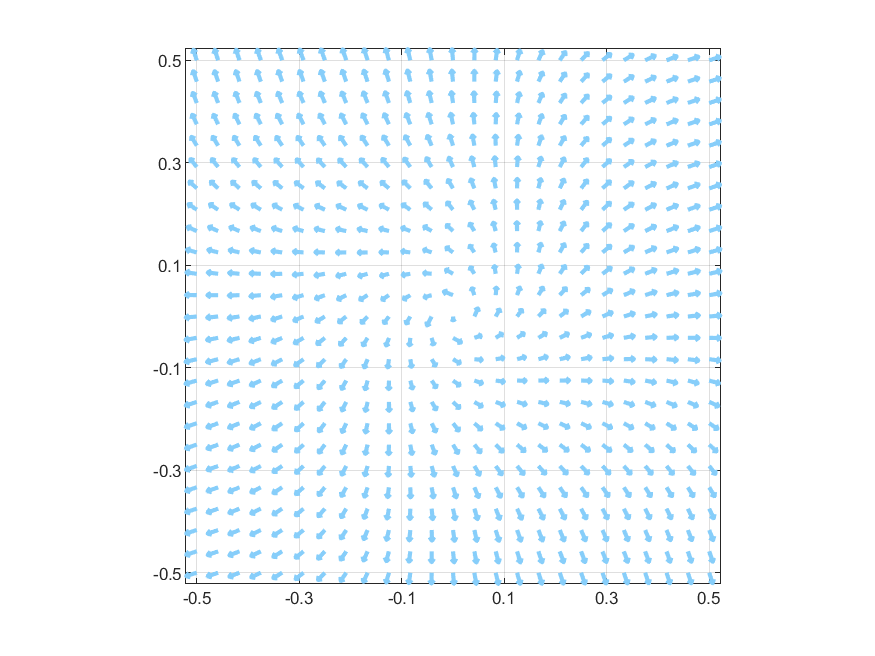}}
		\subfloat[$ t=0.049$.]{\includegraphics[trim=1cm 0cm 1.5cm 0.5cm, clip,width=0.32\linewidth]{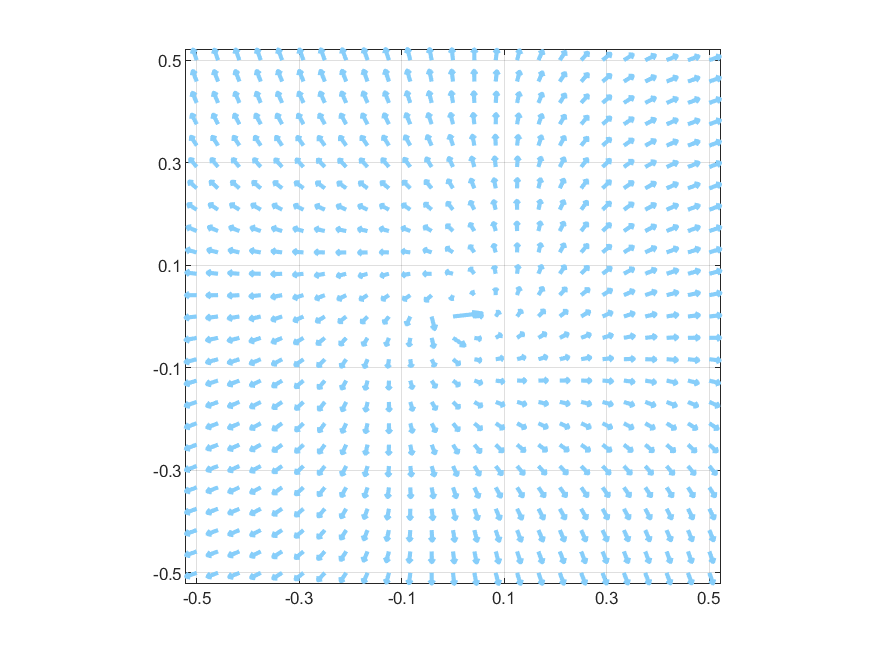}}
		\subfloat[$ t=0.1 $.]{\includegraphics[trim=1cm 0cm 1.5cm 0.5cm, clip,width=0.32\linewidth]{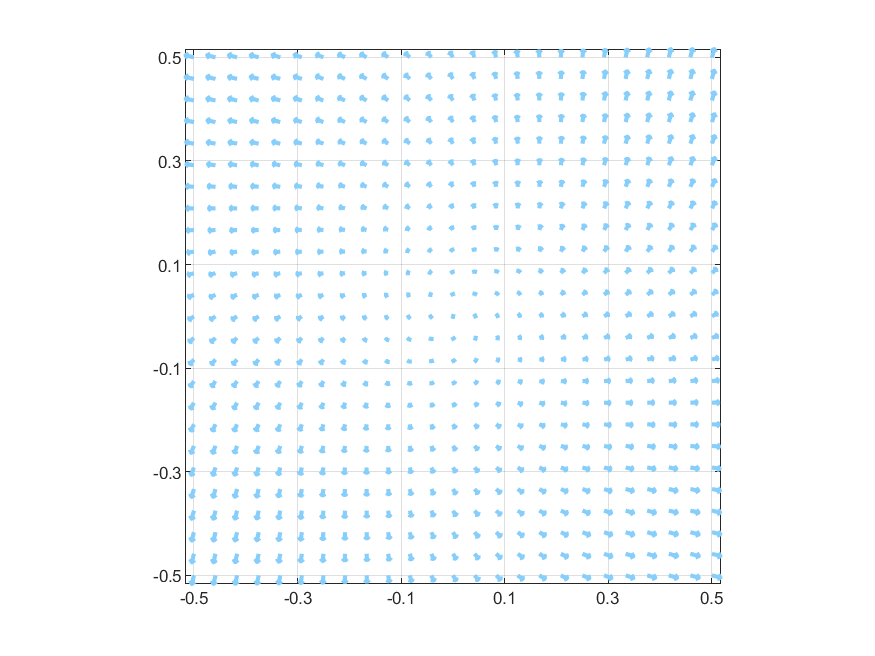}}
		\caption{ Example \ref{ex_2}: Numerical solutions $\bm{m}$ at $t = 0,\, 0.02,\, 0.04,\, 0.048,\, 0.049,\, 0.1$ using the PRK2 scheme with $\tau = 10^{-4}$.\label{fig:LLG-top view}}
	\end{figure}
	
	\begin{figure}[!htbp]
		\centering	
		\subfloat[ Initial condition.]{\includegraphics[trim=1cm 0cm 1.5cm 0.5cm, clip,width=0.32\linewidth]{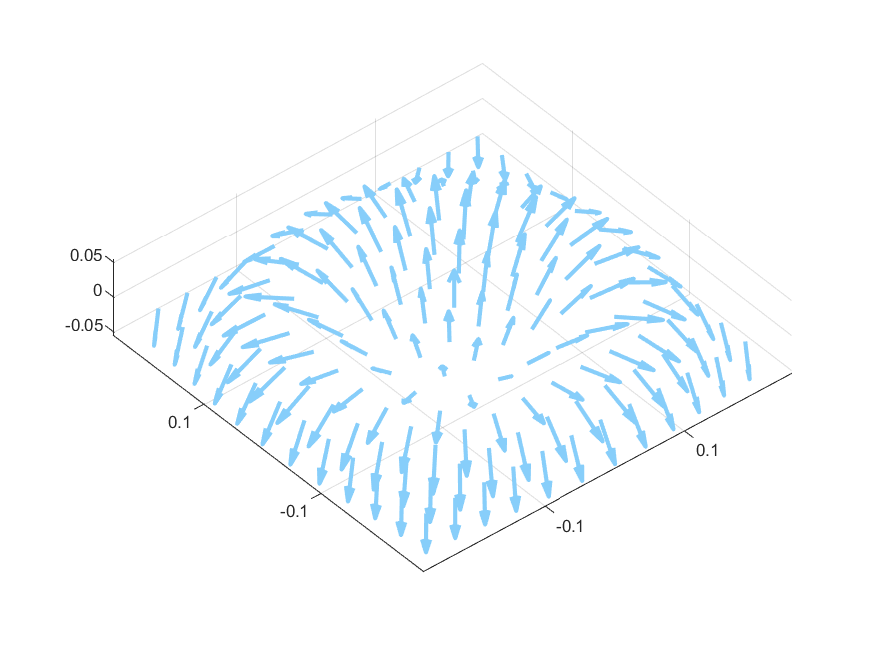}}
		\subfloat[$ t=0.049 $.]{\includegraphics[trim=1cm 0cm 1.5cm 0.5cm, clip,width=0.32\linewidth]{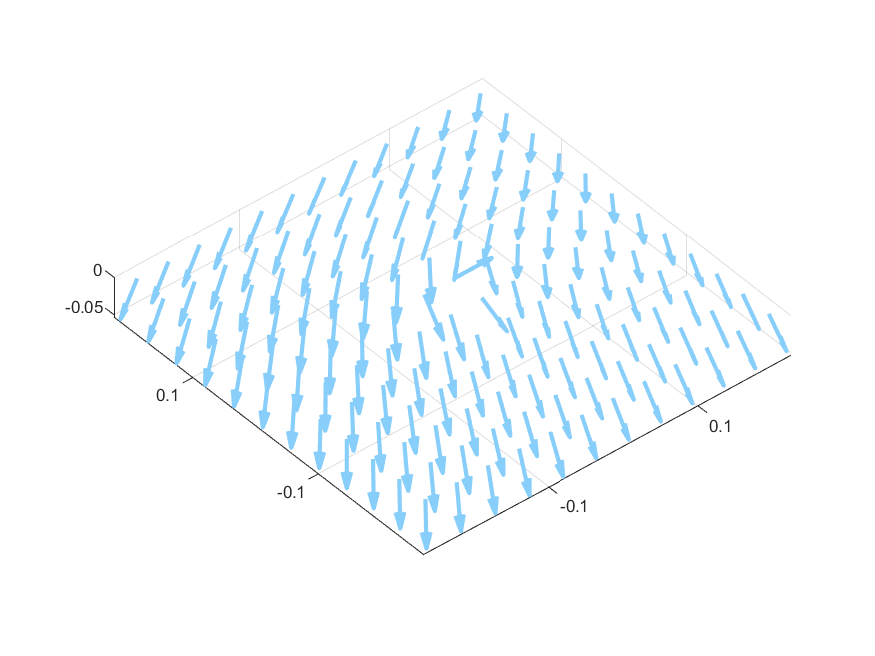}}
		\subfloat[$ t=0.1 $.]{\includegraphics[trim=1cm 0cm 1.5cm 0.5cm, clip,width=0.32\linewidth]{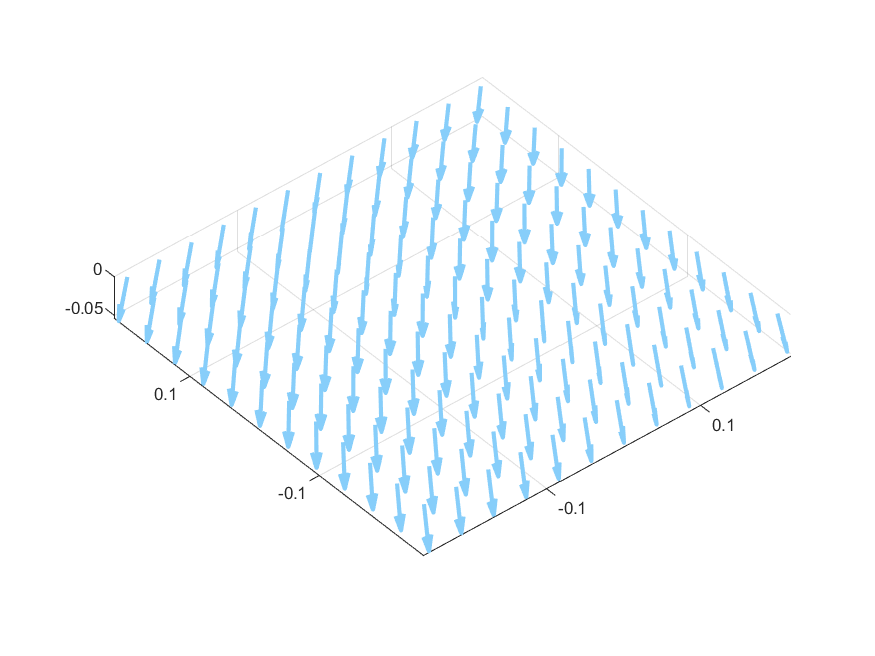}}
		\caption{ Example \ref{ex_2}: Numerical solutions $\bm{m}$ for nodals close to the origin at $t = 0,\, 0.049,\, 0.1$ using the PRK2 scheme with $\tau = 10^{-4}$. \label{fig:LLG-front view}}
	\end{figure}

To assess the accuracy and computational efficiency of the proposed PRK schemes, we perform comparative experiments with the SIP1 and LM2 schemes. All methods are applied to equation~\eqref{LLG} using identical parameters
$\alpha = \beta = 1$, spatial mesh size $h = 1/48$, and time step $\tau = 10^{-4}$.

Figure~\ref{fig:compare_all}(a) shows the discrete energy evolution produced by four schemes at the relatively large time step $\tau = 10^{-4}$. Although both SIP1 and LM2 remain stable at this time step, they fail to accurately predict the transition dynamics when compared with the reference solution obtained by a fourth-order BDF scheme with $\tau = 10^{-6}$. Specifically, the first-order SIP1 scheme exhibits a delayed energy drop, whereas the second-order LM2 method suffers from a premature energy drop. Both behaviors indicate an inaccurate resolution of the physical evolution time. In contrast, the energy curves produced by PRK2 and the
alternative formulation~\eqref{other-PRK2} are nearly indistinguishable
and accurately reproduce the transition time, closely following the reference energy evolution. These results indicate that the PRK framework
maintains higher accuracy at larger time steps while preserving the correct dissipation behavior. Although a rigorous proof of energy
dissipation for scheme~\eqref{other-PRK2} is left for future work, its
numerical performance is virtually identical to that of the proven PRK2 scheme, underscoring the intrinsic accuracy of product-type IMEX
Runge--Kutta methods.

To further evaluate computational efficiency, work--precision diagrams
are presented at three representative terminal times,
$T = 0.02$, $0.08$, and $0.2$, as shown in
Figure~\ref{fig:compare_all}(b)--(d). These plots illustrate the balance
between computational cost (CPU time) and achieved accuracy.

Compared with SIP1, PRK2 consistently demonstrates higher efficiency
across all tested time intervals. Despite its higher per-step cost due to
the two-stage structure, the ability to employ larger stable time steps
more than compensates for this overhead, resulting in improved
work--precision performance.

In comparison with LM2, the relative efficiency depends on the stage of
the dynamics. At early times ($T = 0.02$), LM2 is more efficient owing to
its constant-coefficient linear solvers. However, as the solution
approaches the transition regime (around $T = 0.08$, corresponding to the
energy drop in Figure~\ref{fig:compare_all}(a)), the efficiency of LM2
deteriorates significantly, reflecting difficulties in resolving the
transition timing. In contrast, PRK2 maintains its expected second-order
convergence and becomes markedly more efficient, an advantage that
persists into the later stage of the simulation ($T = 0.2$).

\begin{figure}[!htbp]
	\centering
	\subfloat[Energy evolution.]{
\includegraphics[width=0.43\linewidth]{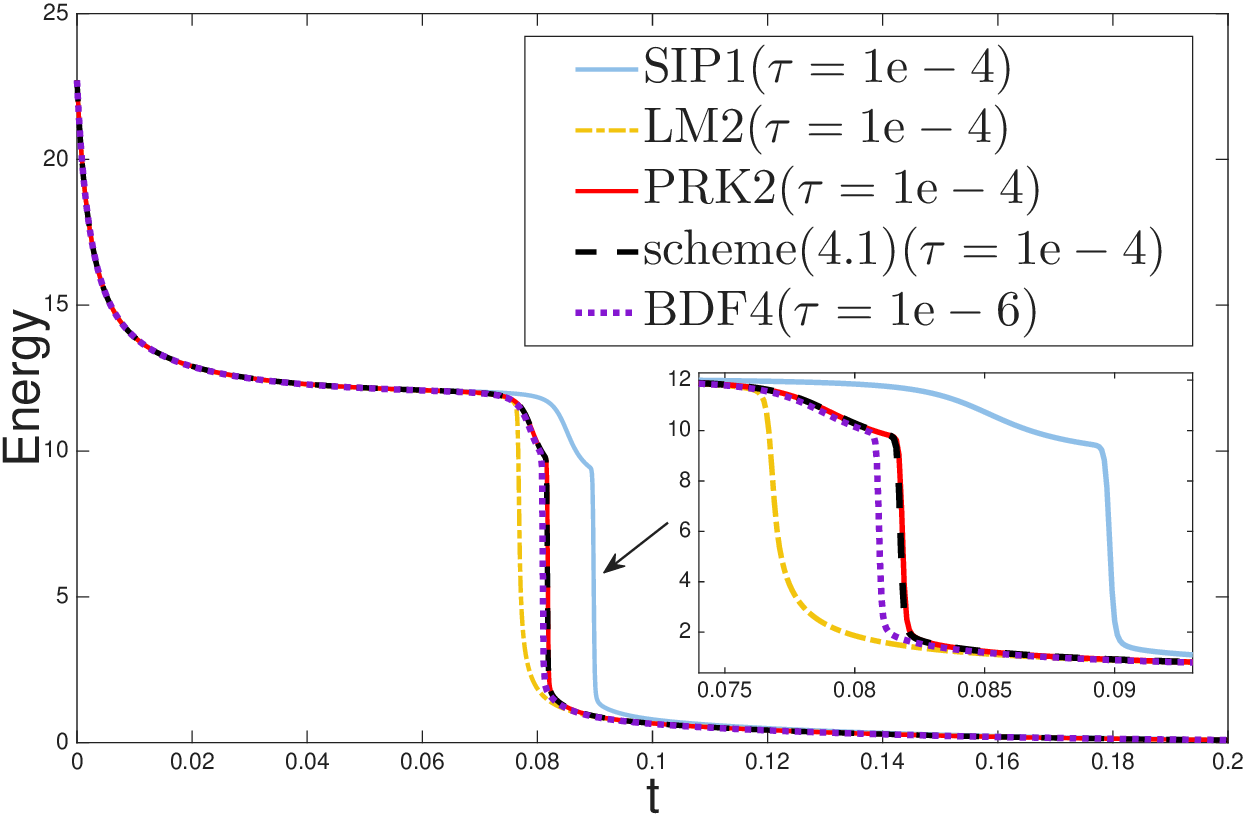}
	}
	\subfloat[$L^2$ error v.s. CPU times(s) at $T=0.02$.]{
		\includegraphics[width=0.43\linewidth]{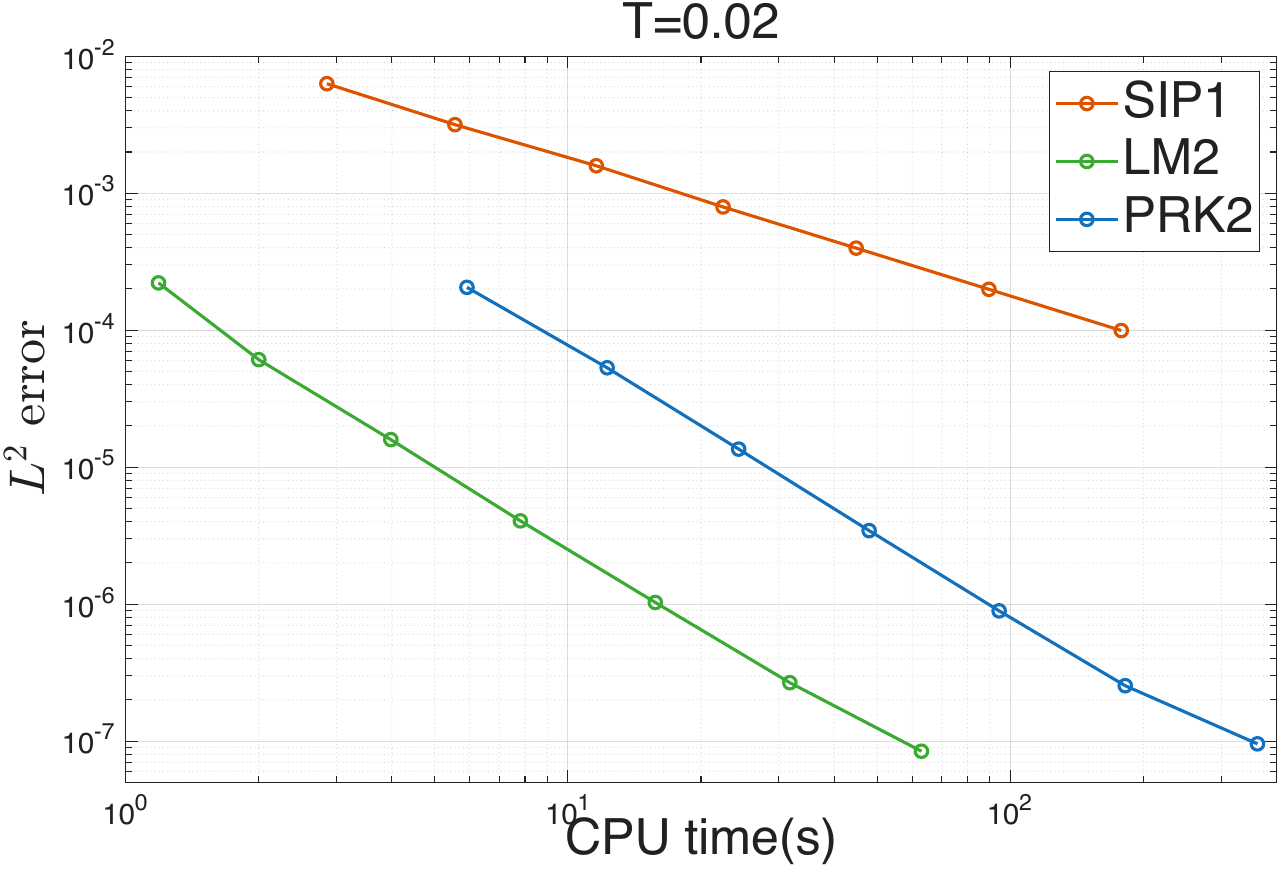}
	}\\
	\subfloat[$L^2$ error v.s. CPU times(s) at $T=0.08$.]{
		\includegraphics[width=0.43\linewidth]{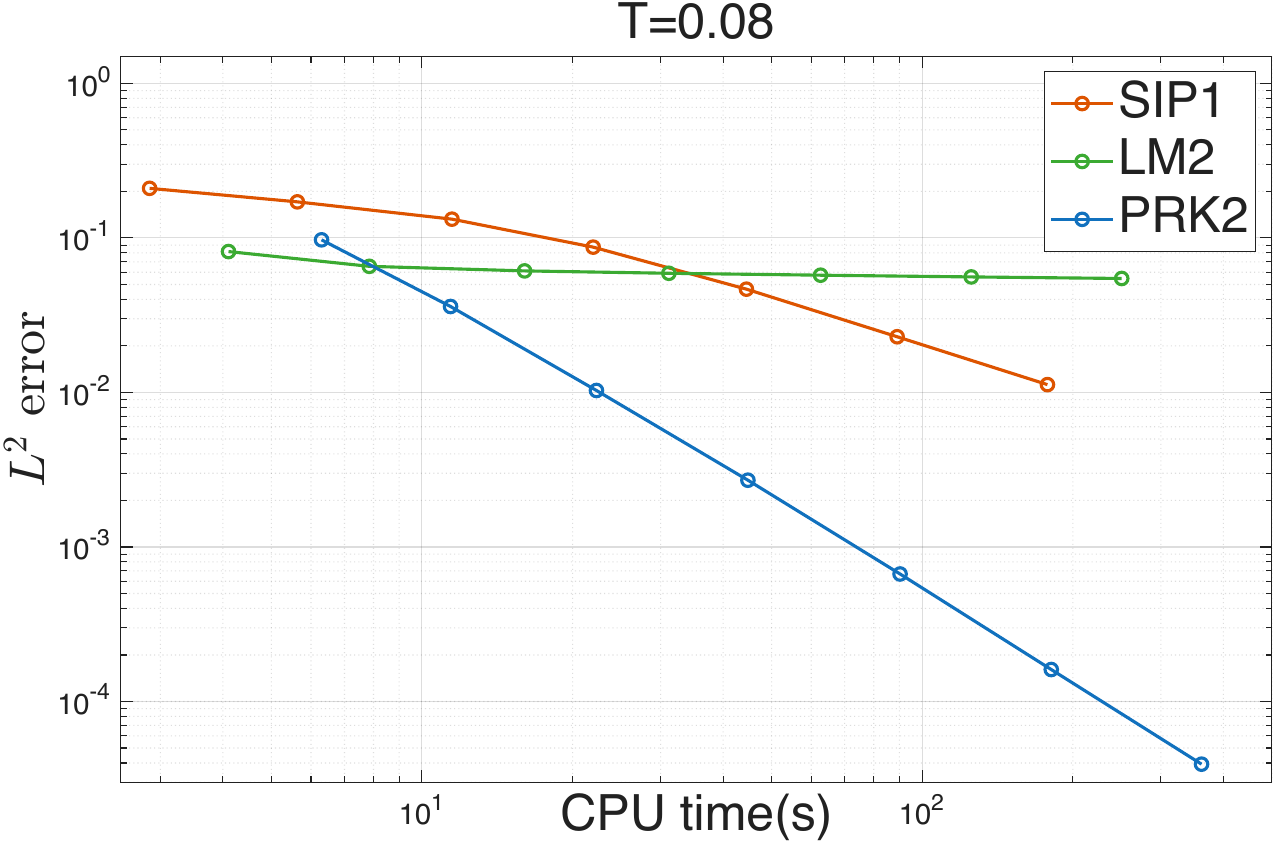}
	}
	\subfloat[$L^2$ error v.s. CPU times(s) at $T=0.2$.]{
		\includegraphics[width=0.43\linewidth]{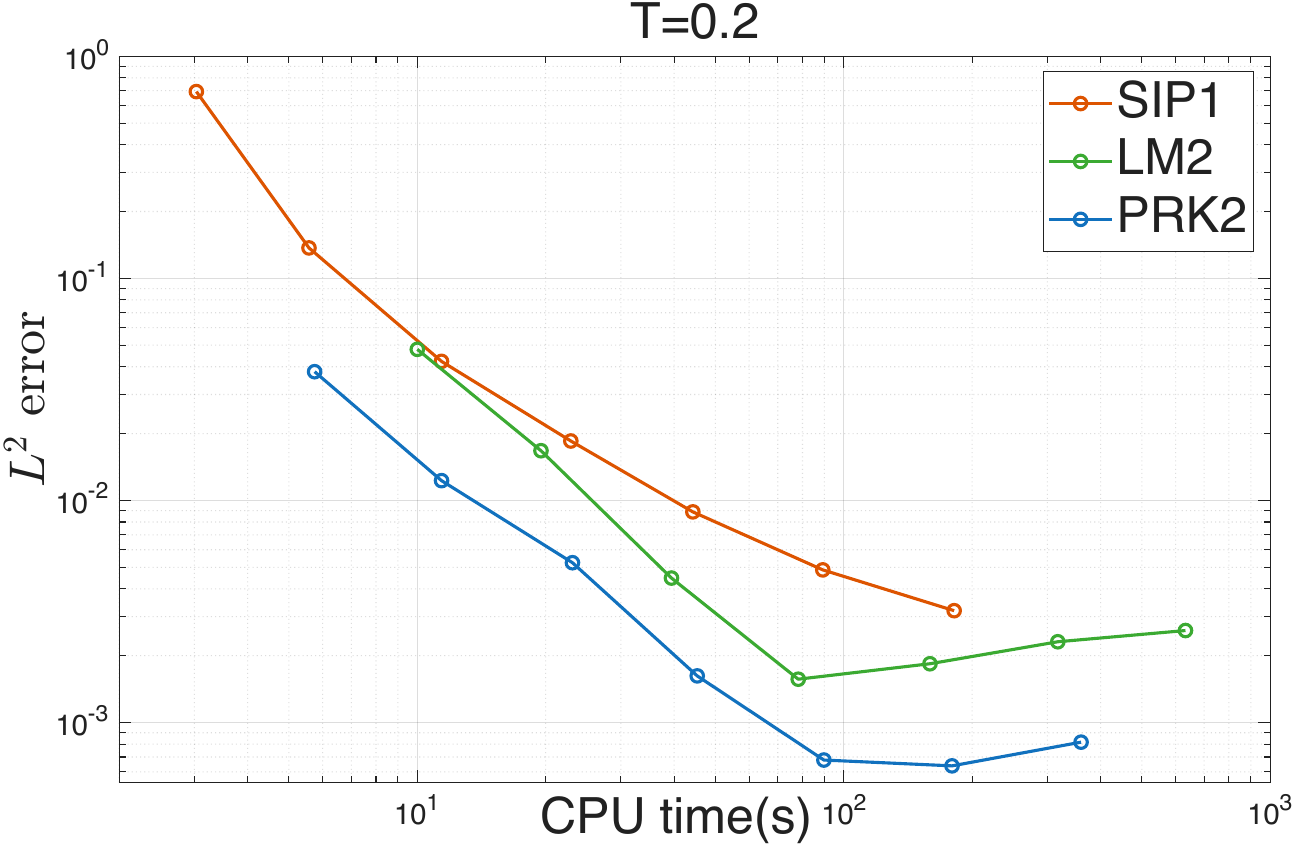}
	}
	\caption{Example~\ref{ex_2}: 
	Comparison of energy evolution and computational efficiency for SIP1,
	LM2, and the proposed PRK scheme. The work-precision diagrams are generated using a fixed spatial mesh size $h = 1/48$ and a sequence of time steps $\tau = 2\times 10^{-3} / 2^j$ for $j = 1, 2, \dots, 7$. 
    }
	\label{fig:compare_all}
\end{figure}

Table~\ref{tab:errors} reports the $L^2$ errors of LM2 and PRK2 for various
final times $T$ and time step sizes $\tau$. A pronounced difference in
robustness is observed. The LM2 scheme is highly sensitive to the time
step: for relatively large $\tau$ (e.g., $\tau = 10^{-3}$), the nonlinear
Lagrange multiplier problem may fail to admit a real solution, leading to
breakdown of the method (denoted by `NAN') at early times such as
$T = 0.004$. As the simulation progresses, LM2 becomes inapplicable even
for smaller time steps.

In contrast, both the PRK2 scheme and scheme~\eqref{other-PRK2} remain
robust across all tested time steps and final times, including long-time
simulations up to $T = 0.2$. This stable performance is consistent with
the linear structure of the PRK framework, which guarantees unique
solvability at each discrete time step and supports reliable numerical
behavior even for relatively large time steps.

\begin{table}[!htbp]
\centering
\caption{Robustness test: Errors at different final times $T$ and time steps $\tau$.}
\label{tab:errors}
\begin{tabular}{c cc cc cc}
\toprule
& \multicolumn{2}{c}{LM2} & \multicolumn{2}{c}{PRK2}& \multicolumn{2}{c}{scheme \eqref{other-PRK2}} \\
\cmidrule(lr){2-3} \cmidrule(lr){4-5} \cmidrule(lr){6-7}
$T\backslash\tau$
& $10^{-3}$  & $2\mathrm{e}{-4}$
& $10^{-3}$ & $2\mathrm{e}{-4}$
& $10^{-3}$ & $2\mathrm{e}{-4}$\\
\midrule
0.002
& $1.24\mathrm{e}{-1}$  & $1.29\mathrm{e}{-3}$
& $8.07\mathrm{e}{-3}$ & $5.53\mathrm{e}{-4}$ & $7.86\mathrm{e}{-3}$ & $5.42\mathrm{e}{-4}$ \\
0.004
& NAN & $6.22\mathrm{e}{-3}$
& $1.02\mathrm{e}{-2}$ & $6.43\mathrm{e}{-4}$ & $1.00\mathrm{e}{-2}$ & $6.36\mathrm{e}{-4}$\\
0.006
& -- & $9.84\mathrm{e}{-1}$
& $1.09\mathrm{e}{-2}$ & $6.64\mathrm{e}{-4}$ & $1.09\mathrm{e}{-2}$ & $6.63\mathrm{e}{-4}$\\
0.008
& -- & NAN
& $1.13\mathrm{e}{-2}$ & $6.72\mathrm{e}{-4}$ &$1.13\mathrm{e}{-2}$ & $6.75\mathrm{e}{-4}$\\
0.12
& -- & --
& $8.42\mathrm{e}{-2}$ & $6.32\mathrm{e}{-3}$ & $8.31\mathrm{e}{-2}$ & $6.29\mathrm{e}{-3}$\\
0.2
& -- & --
& $3.80\mathrm{e}{-2}$ & $4.02\mathrm{e}{-3}$ & $3.73\mathrm{e}{-2}$ & $4.01\mathrm{e}{-3}$\\
\bottomrule
\end{tabular}
\end{table}

	\subsection{The Oseen-Frank model} 
	We investigate the physically relevant one-constant
	approximation of the Oseen-Frank model \cite{bartels2005stability}, which reduces the problem to finding a harmonic minimizing map with values constrained on the unit sphere in \( \mathbb{R}^3 \). The governing equation is given by:
	\begin{equation}\label{OF}
		\begin{aligned}
			&\bm{m}_t = \left(I - \bm{m}\bm{m}^\top  /|\bm{m}^2|\right)\Delta\bm{m},\\
			&|\bm{m}|=1.
		\end{aligned}
	\end{equation}
	To demonstrate the robustness of our proposed scheme, we perform numerical experiments in the cubic domain \( \Omega = (0,1)^3 \).
	
\begin{example}\label{ex_3_1}{\rm
	We simulate the evolution of a point defect. Dirichlet boundary conditions are imposed on \( \partial\Omega \) as:  
	\[
	\bm{m} = \frac{(x,y,z) - (0.5, 0.5, 0.5)}{|(x,y,z) - (0.5, 0.5, 0.5)|}.
	\]  	
	The initial condition is specified as:  
	\begin{align*}
		m_1(x,y,z,0) &= \sin(2\pi x) + 2 + 0.5\sin(6\pi y) + 0.2\sin(4\pi z), \\
		m_2(x,y,z,0) &= \cos(2\pi x) + 2 + 0.5\cos(6\pi y) + 0.2\cos(4\pi z), \\
		m_3(x,y,z,0) &= \sin(2\pi x) + 6\cos(6\pi y) + \cos(4\pi z).
	\end{align*} } 	
\end{example}	

	We apply our scheme with \( h = 1/24 \) and \( \tau = 10^{-3} \) to compute numerical approximations of equation \eqref{OF}. Figure~\ref{fig:Example 2(3D)} illustrates the evolution of the director field \( \bm{m} \) toward the formation of a point defect at the center of the domain. For clarity, only selected slices of the solution are displayed in Figure~\ref{fig:Example 2(3D)}, while Figure~\ref{fig:Example 2(3D)-proj} shows the projection view of \( \bm{m} \) onto the plane \( \{(x,y,z) \in \Omega : z = 0.5\} \).  
	
	The results reveal that our scheme rotates the vectors, leading to the formation of a single degree-$1$ singularity within just a few iterations. As the simulation progresses, the point defect gradually migrates toward the center and eventually stabilizes in a steady state.
	
	\begin{figure}[!htbp]
		\centering	
		\subfloat[$ t=0 $.]{\includegraphics[trim=1.5cm 0.5cm 1.5cm 0.5cm, clip,width=0.43\linewidth]{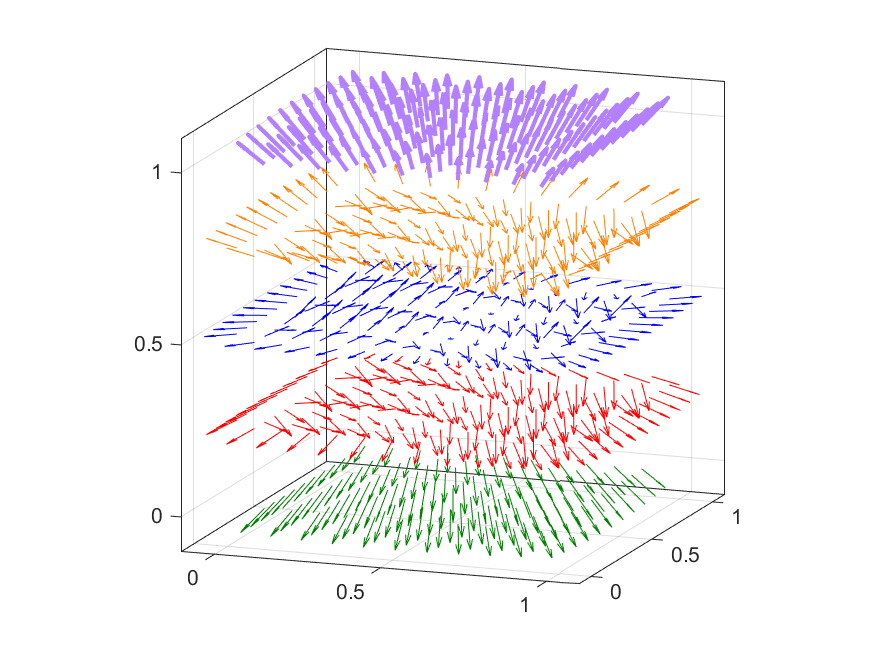}}
		\subfloat[$ t=0.015 $.]{\includegraphics[trim=1.5cm 0.5cm 1.5cm 0.5cm, clip,width=0.43\linewidth]{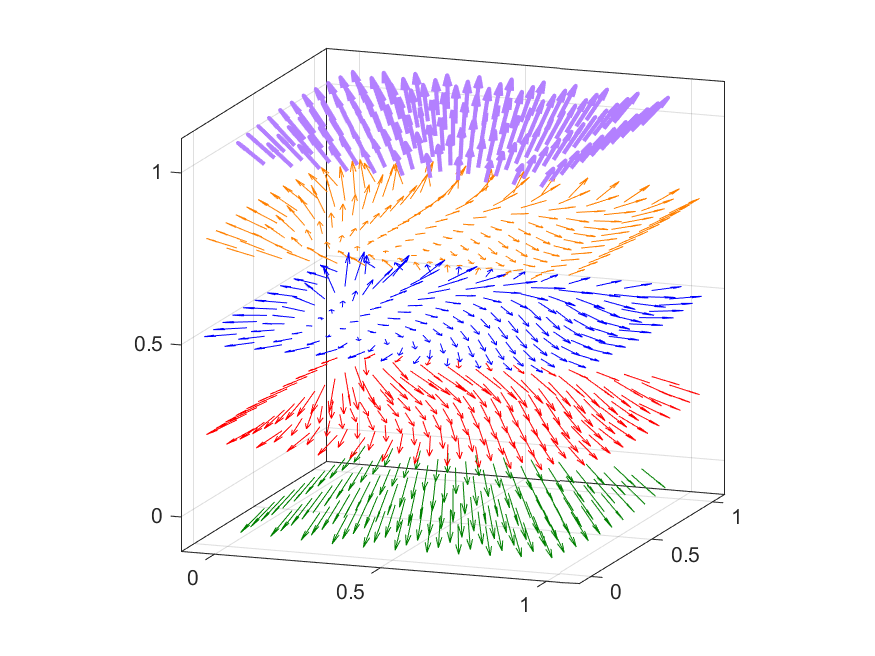}}\\
		\subfloat[$ t=0.05$.]{\includegraphics[trim=1.5cm 0.5cm 1.5cm 0.5cm, clip,width=0.43\linewidth]{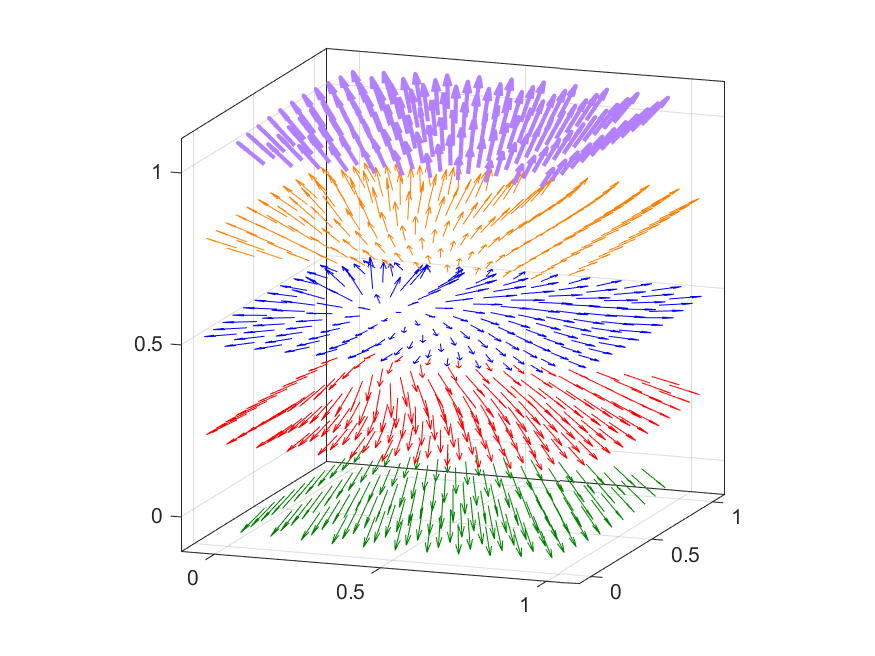}}
		\subfloat[$ t=0.6 $.]{\includegraphics[trim=1.5cm 0.5cm 1.5cm 0.5cm, clip,width=0.43\linewidth]{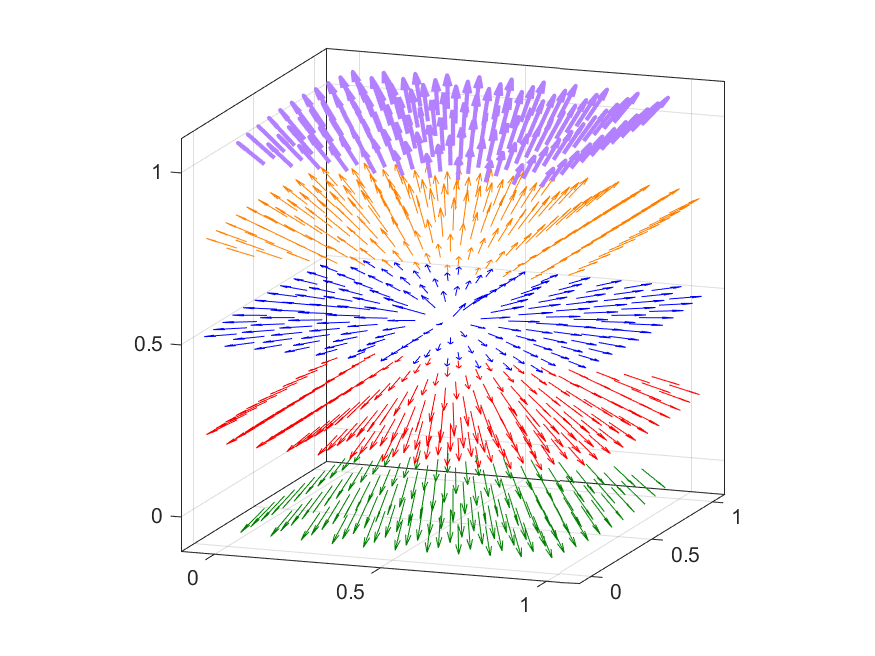}}
		\caption{ Example \ref{ex_3_1}: Numerical solutions $\bm{m}$ at $t = 0,\, 0.015,\, 0.05,\, 0.6$ computed using the PRK2 scheme with $\tau = 10^{-3}$. \label{fig:Example 2(3D)}}
	\end{figure}
	
	\begin{figure}[!htbp]
		\centering	
		\subfloat[$ t=0 $.]{\includegraphics[trim=1.5cm 0.5cm 1.5cm 0.5cm, clip,width=0.43\linewidth]{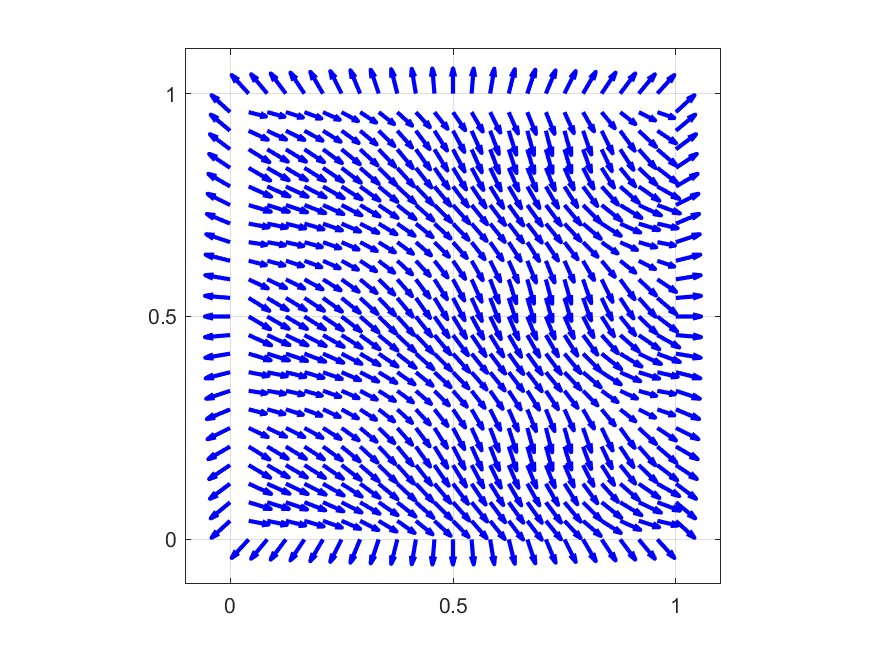}}
		\subfloat[$ t=0.015 $.]{\includegraphics[trim=1.5cm 0.5cm 1.5cm 0.5cm, clip,width=0.43\linewidth]{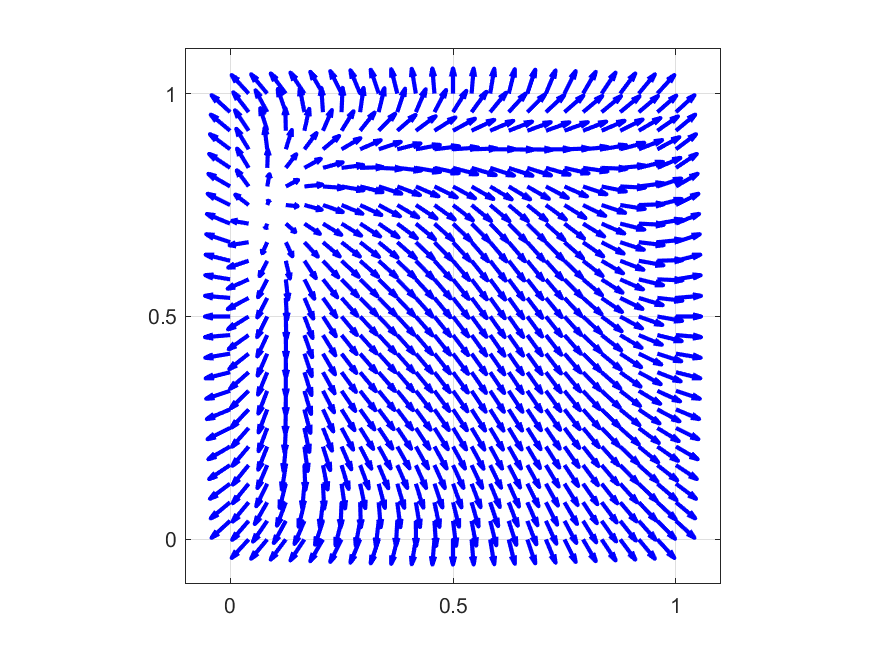}}\\
		\subfloat[$ t=0.05$.]{\includegraphics[trim=1.5cm 0.5cm 1.5cm 0.5cm, clip,width=0.43\linewidth]{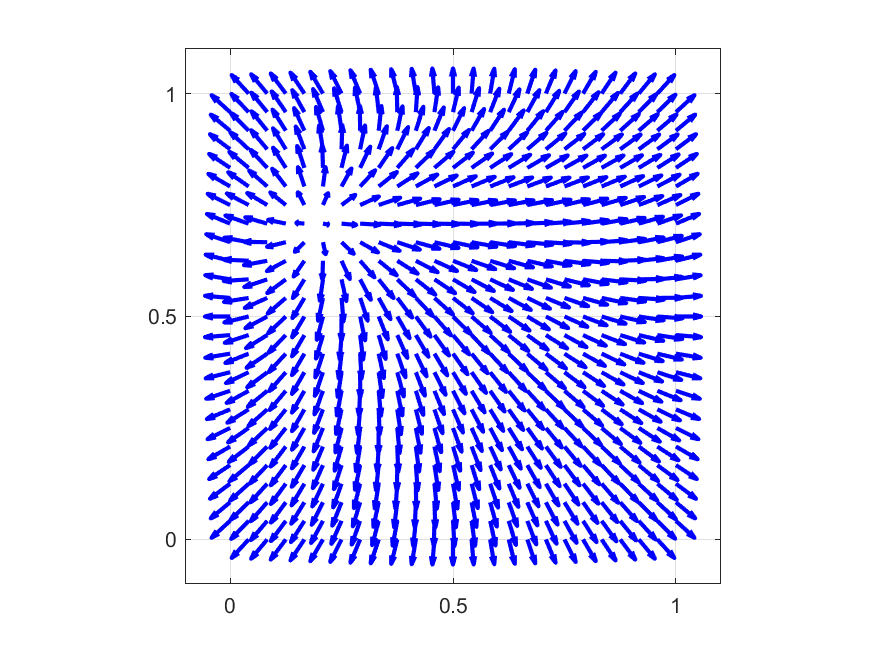}}
		\subfloat[$ t=0.6 $.]{\includegraphics[trim=1.5cm 0.5cm 1.5cm 0.5cm, clip,width=0.43\linewidth]{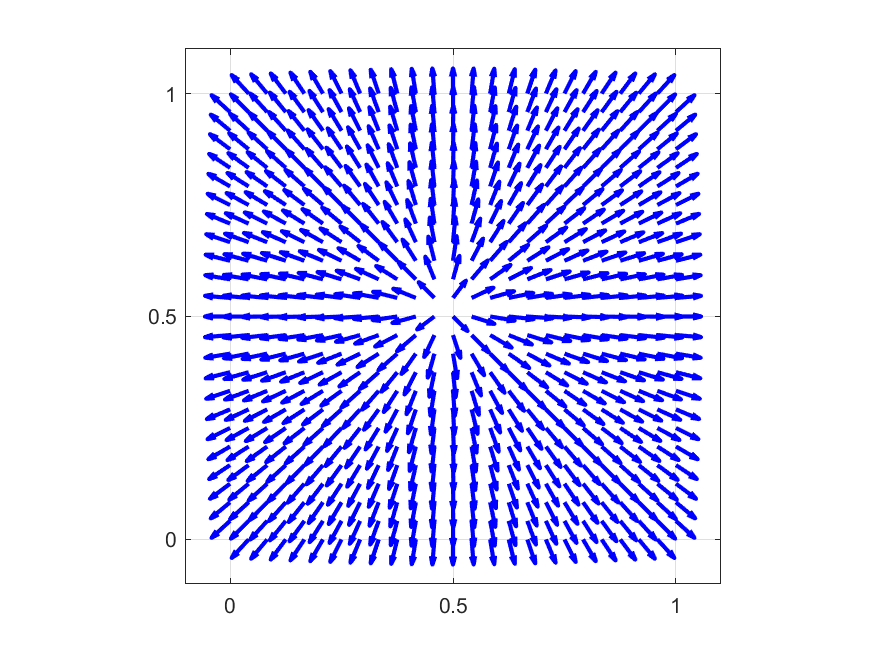}}
		\caption{ Example \ref{ex_3_1}: Numerical solutions $\bm{m}$ at $\{(x,y,z)\in\Omega:z=0.5\}$ for $t = 0,\, 0.015,\, 0.05,\, 0.6$ computed using the PRK2 scheme with $\tau = 10^{-3}$. \label{fig:Example 2(3D)-proj}}
	\end{figure}

\begin{example}\label{ex_3_2}{\rm
	We simulate a twisted nematic state of liquid crystals. Dirichlet boundary conditions are enforced on the planes $z=0$ and $z=1$, while homogeneous Neumann boundary conditions are imposed on the remaining boundaries. Specifically: on the lower slab $\{(x,y,z) \in \Omega : z = 0\}$, the nematic
	rods are aligned parallel to the $x$-axis with $\bm{m}=(1,0,0)^\top $. For the upper slab $\{(x,y,z) \in \Omega : z =0.5\}$, the rods are uniformly aligned along the $y$-axis with $\bm{m}=(0,1,0)^\top $.} 	
\end{example}	

	We initialize the computation with random unit vectors and apply our scheme with \( h = 1/24 \) and \( \tau = 5\times10^{-3} \). Figure~\ref{fig:Example 3.2-1} shows the evolution of the nematic configuration over time. The algorithm efficiently transitions the initially disordered configuration into a more stable twisted configuration. For clarity, Figure~\ref{fig:Example 3.2-2} provides a close-up view of the numerical solution along the	$z$-axis at $x=y=0.5$. The expected twisted equilibrium solution is clearly observed.
	
	\begin{figure}[!htbp]
		\centering	
		\subfloat[$ t=0 $.]{\includegraphics[trim=1.5cm 0.5cm 1.5cm 0.5cm, clip,width=0.43\linewidth]{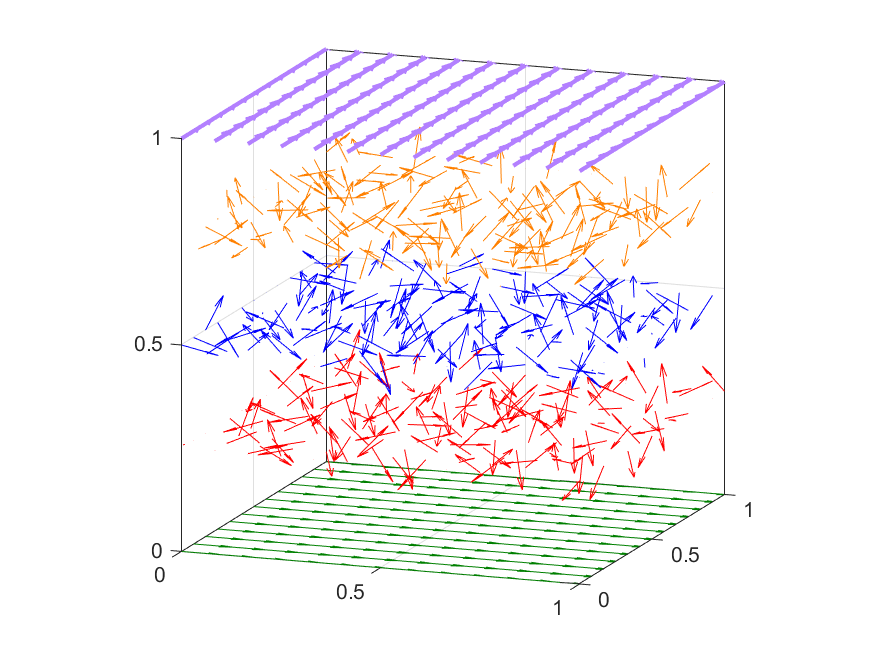}}
		\subfloat[$ t=0.05 $.]{\includegraphics[trim=1.5cm 0.5cm 1.5cm 0.5cm, clip,width=0.43\linewidth]{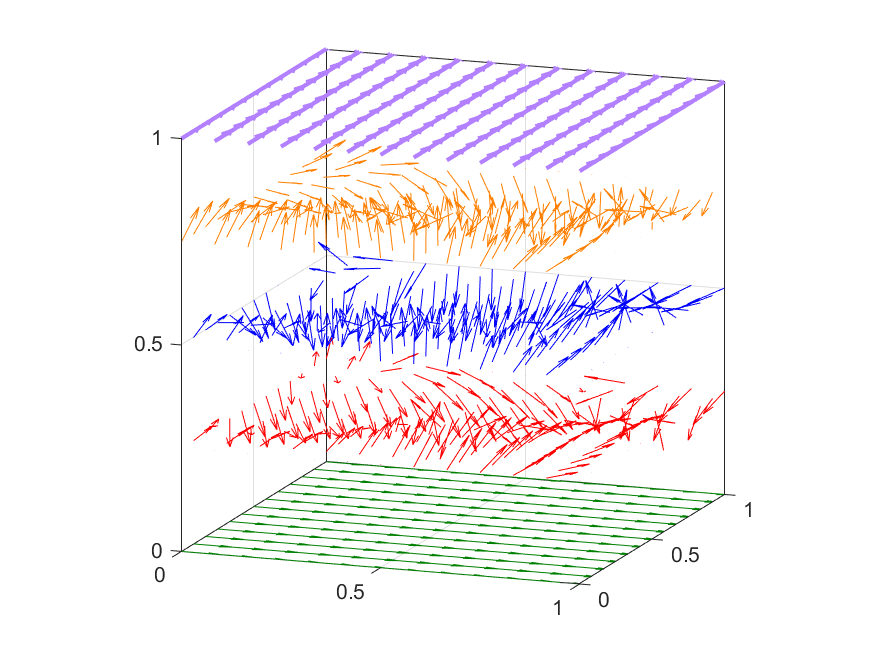}}\\
		\subfloat[$ t=0.15$.]{\includegraphics[trim=1.5cm 0.5cm 1.5cm 0.5cm, clip,width=0.43\linewidth]{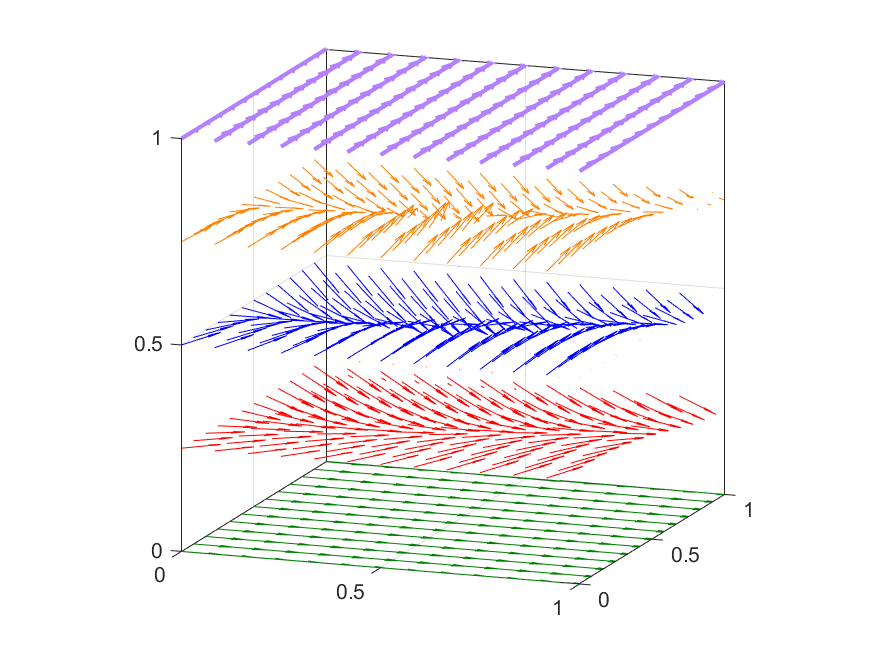}}
		\subfloat[$ t=0.5 $.]{\includegraphics[trim=1.5cm 0.5cm 1.5cm 0.5cm, clip,width=0.43\linewidth]{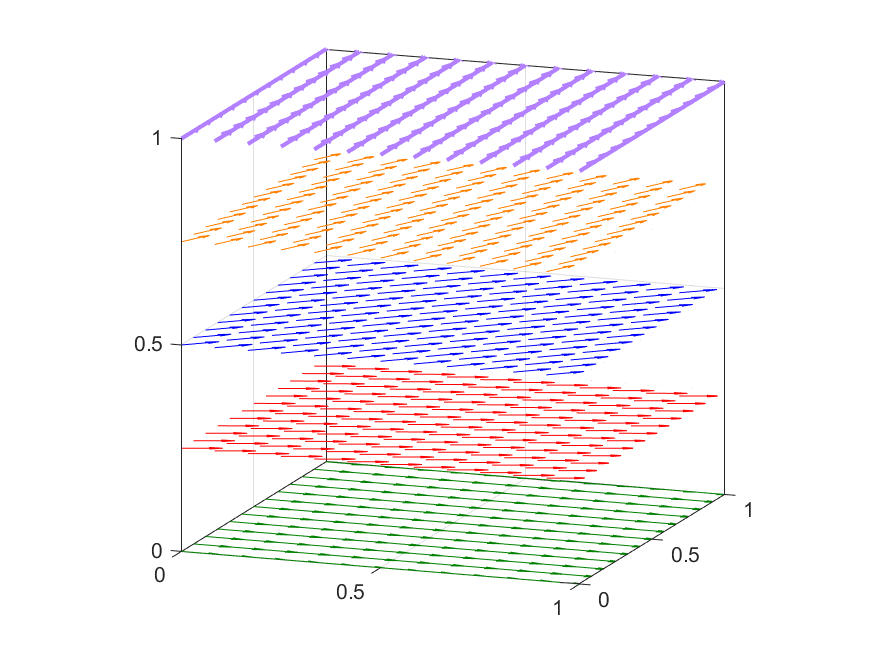}}
		\caption{ Example \ref{ex_3_2}: Numerical solutions $\bm{m}$ at $t = 0,\, 0.05,\, 0.15,\, 0.5$ computed using the PRK2 scheme with $\tau = 5\times10^{-3}$. \label{fig:Example 3.2-1}}
	\end{figure}
	
	\begin{figure}[!htbp]
		\centering	
		\subfloat[$ t=0 $.]{\includegraphics[trim=5cm 0.5cm 5cm 0.5cm, clip,width=0.2\linewidth]{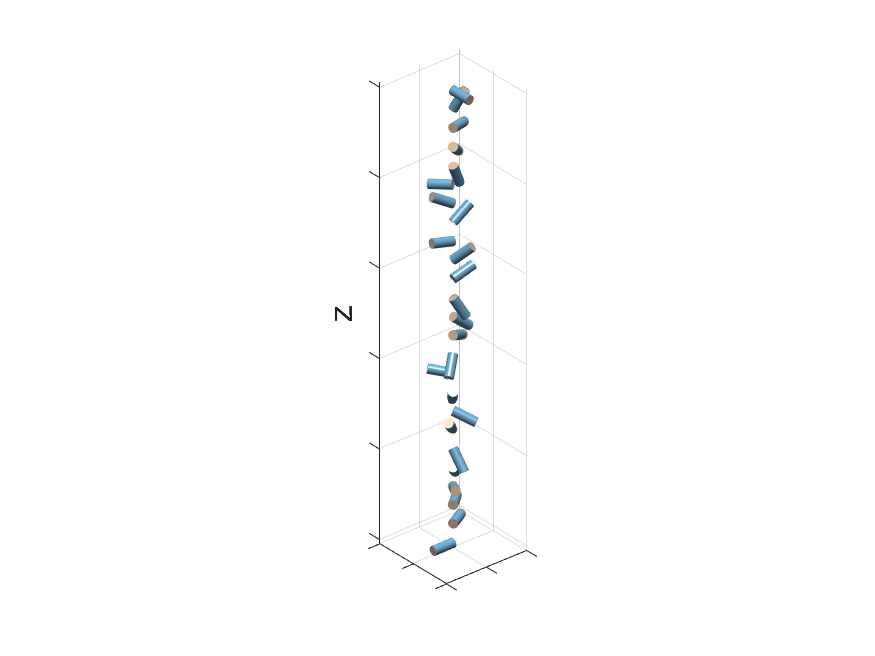}}
		\subfloat[$ t=0.05 $.]{\includegraphics[trim=5cm 0.5cm 5cm 0.5cm, clip,width=0.2\linewidth]{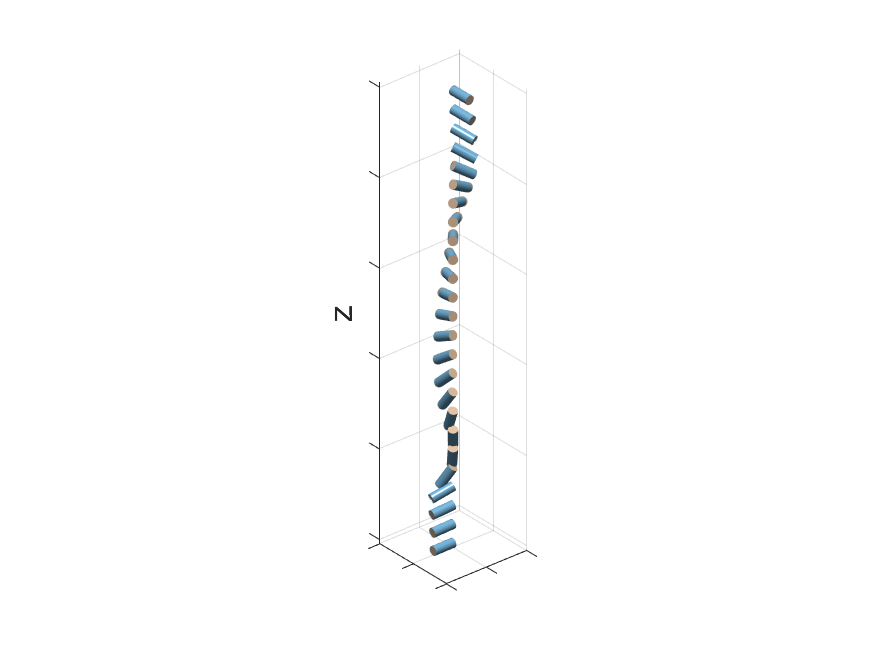}}
		\subfloat[$ t=0.15$.]{\includegraphics[trim=5cm 0.5cm 5cm 0.5cm, clip,width=0.2\linewidth]{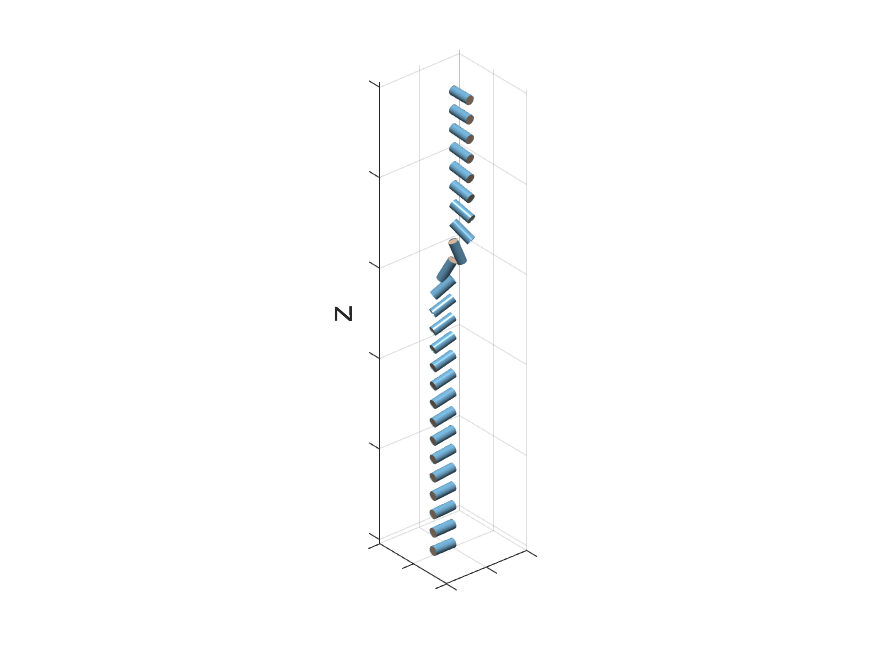}}
		\subfloat[$ t=0.5 $.]{\includegraphics[trim=5cm 0.5cm 5cm 0.5cm, clip,width=0.2\linewidth]{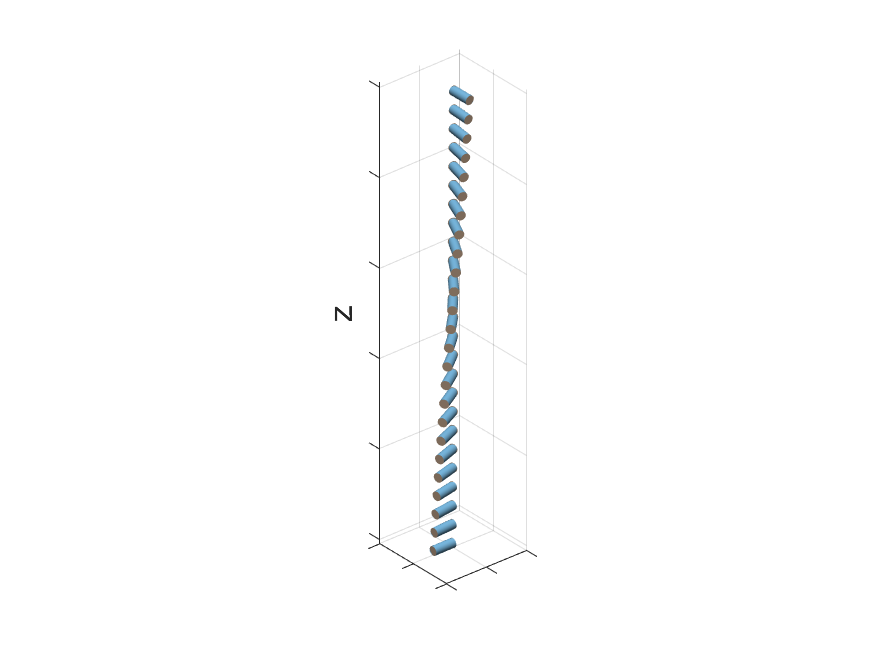}}
		\caption{ Example \ref{ex_3_2}: Numerical solutions $\bm{m}$ at $\{(x,y,z)\in\mathbb{R}^3:x=y=0.5\}$ for $t = 0,\, 0.05,\, 0.15,\, 0.5$ computed using the PRK2 scheme with $\tau = 5\times10^{-3}$. \label{fig:Example 3.2-2}}
	\end{figure}   

\section*{Acknowledgments}
This work is supported by the National Science Foundation of China (No.12271240, 12426312), the fund of the Guangdong Provincial Key Laboratory of Computational Science and Material Design, China (No.2019B030301001), and the Shenzhen Natural Science Fund (RCJC20210609103819018). The third author is partially supported by Zhuhai Innovation and Entrepreneurship Team Project (2120004000498).

\bibliographystyle{siamplain}
\bibliography{references}
\end{document}